\documentclass[12pt]{article}
\pdfoutput=1
\usepackage[utf8]{inputenc}
\usepackage{amssymb,latexsym,graphicx}
\usepackage{amsmath,amscd,amsthm}
\usepackage{mathrsfs}
\usepackage{amsrefs}
\usepackage{graphicx}
\usepackage{esint}
\usepackage{subcaption}
\usepackage{breqn}
\usepackage{tikz}
\usepackage{tikz-cd} 
\usetikzlibrary{arrows,decorations.markings}
\usepackage{environ}
\usepackage{pgfplots}
\usepackage{gensymb}
\usepackage{math tools}
\usepackage[letterpaper, margin=1 in]{geometry}	
\usepackage{easybmat}
\usepackage{multirow,bigdelim}
\numberwithin{equation}{section}
\theoremstyle{plain}
\usepackage{relsize}
\usepackage{hyperref}
\hypersetup{
    colorlinks=true,
    linkcolor=black,
    filecolor=black,      
    urlcolor=black,
    citecolor=black,
}
\usepackage{pgfplots}
\usepgfplotslibrary{fillbetween}
\usetikzlibrary{patterns}

\newtheorem{thm}{Theorem}[section]
\newtheorem*{thm*}{Theorem}

\newtheorem{cor}[thm]{Corollary}
\newtheorem{remark}[thm]{Remark}
\newtheorem{defn}[thm]{Definition}
\newtheorem{lemma}[thm]{Lemma}

\usepackage[bottom]{footmisc}

\newcommand{\diff}{\mathrm{d}}

\newcommand{\RN}{\mathbb{R}^{n}}

\newcommand{\T}{\mathrm{T}}

\DeclareMathOperator{\sgn}{sgn} 

\newcommand{\Y}{\mathcal{Y}}
\newcommand{\cbx}{C_{b}\left(X\right)}

\definecolor{indigo(web)}{rgb}{0.29, 0.0, 0.51}
\definecolor{islamicgreen}{rgb}{0.0, 0.56, 0.0}
\definecolor{ballblue}{rgb}{0.13, 0.67, 0.8}
\definecolor{alizarin}{rgb}{0.82, 0.1, 0.26}

\AtBeginDocument{%
   \def\MR#1{}
}
\title{A Mass-Conserving Formulation of the Generalized Benjamin-Bona-Mahony-Burgers Equation on Star Networks}
\author{A. George Morgan\footnote{Department of Mathematics,
University of Toronto,
40 St. George St., Room 6290,
Toronto, Ontario, CA,
M5S 2E4. Institutional email: 
adam.morgan@mail.utoronto.ca}
}
\date{\today}
\begin{document}
\maketitle
\abstract
The generalized Benjamin-Bona-Mahony-Burgers equation (gBBMB) describes the flow of blood through a long, viscoelastic artery. In this article we introduce a formulation of gBBMB valid on networks with semi-infinite edges joined at a single junction, with the network's edges corresponding to a segment of the arterial tree. To reflect sudden changes in the material properties of blood vessels, the coefficients of gBBMB are allowed to take different values on each edge of the network. Critically, our formulation ensures that the total mass of the solution to gBBMB is constant in time, even in the presence of dissipation. We also establish local-in-time well-posedness of this new formulation for sufficiently regular initial data. Then, we show how energy methods can be used to extend the local solution to a solution valid for all positive times, provided certain constraints are imposed on the parameters of the model PDE and the network. To build intuition for how waves scatter off the central junction of a network with two edges, we demonstrate the results of some numerical simulations. 
\section{Introduction}
Let $u(x,t)$ represent the deviation from equilibrium of the cross-sectional area of an artery, where $x$ is the axial  coordinate along the artery and $t$ is time. We assume that the artery is impermeable and viscoelastic, and that the blood it conducts is homogeneous and inviscid. Additionally, we suppose the artery is very long so we can treat $x\in \mathbb{R}$. For some constants $\mu>0 ,\alpha, \nu\geq 0,\ \gamma\in [0,1],$ and $p\in \mathbb{N}$, we model $u(x,t)$ as the solution to the generalized Benjamin-Bona-Mahony-Burgers equation (gBBMB),
\begin{equation}\label{eqn:gbbmb}
\left(1-\mu^2\partial_{xx}\right)u_{t} + \partial_{x}\left(\alpha u + \frac{\gamma}{p+1}u^{p+1}\right) -\nu u_{xx} =0.
\end{equation}
The parameter $\mu$ represents the dispersive influence of the arterial wall's  linear elasticity, $\alpha$ represents the influence of linear advection, $\gamma$ represent the influence of nonlinear advection as well as nonlinear wall elasticity, and $\nu$ represents the influence of viscoelastic dissipation. gBBMB has several alternative names in special cases:
\begin{itemize}
\item{ if $p=1, \nu=0$, \eqref{eqn:gbbmb} is called the Benjamin-Bona-Mahony equation (BBM), introduced by Benjamin et al. \cite{BBM1972};}
\item{ if $p=1$ and $\nu>0$ then \eqref{eqn:gbbmb} is called the Benjamin-Bona-Mahony-Burgers equation (BBMB);}
\item{if $\alpha=\mu=0$ and $p=1$, then \eqref{eqn:gbbmb} is called  Burgers' equation}. 
\end{itemize}
Loosely, we can think of $\mu$ as inversely proportional to the rigidity of the arterial walls. This is reasonable from an intuitive standpoint: if the arterial wall is very rigid then dispersion of a wavepacket travelling down the artery costs a great deal of energy, and so dispersion should have a weak effect on the motion. 
\par Our use of gBBMB as a model for flow in a long, thin-walled viscoelastic tube is motivated by the work of Erbay et al. \cite{EED1992}, where an asymptotic expansion in the primitive fluid-structure interaction equations was used to obtain the generalized Korteweg-de Vries-Burgers equation (gKdVB),
\begin{equation}
\label{eqn:gkdvb}
u_{t} + \mu^{2} u_{xxx} + \partial_{x} \left( \frac{\gamma}{p+1} u^{p+1}\right) -\nu u_{xx} =0,
\end{equation}
as an asymptotic model of blood vessel motion in the cases $p=1,2$; see also \cite{Yomosa1987, RP1979, SS1989, MDLP2019, Cascaval2012} for more on simplified model equations in hemodynamics. If $\nu=0$ and $p=1$, \eqref{eqn:gkdvb} becomes the well-known Korteweg-de Vries equation (KdV). Now,  BBM is well-known as a ``substitute" for KdV (see for example \cite{BPS1983}). In particular for $\nu=0$ both equations support solitary wave solutions, which physiologically correspond to coherent blood pulses. However, gKdVB is third-order in space, while gBBMB is only second-order in space. The high order of gKdVB makes its analysis on subintervals of $\mathbb{R}$ somewhat unnatural. For instance, we must impose three boundary conditions to obtain well-posedness of gKdVB on a finite interval, while gBBMB only requires two boundary conditions. Now, the arterial tree of the human body includes several junctions where a ``parent" artery splits into multiple ``child" arteries or subarteries. If we wish to accommodate these splittings into a blood flow model then boundary effects (more appropriately, junction effects) are extremely important, hence this well-posedness issue for gKdVB becomes very relevant. Additionally, the numerical discretization of second-order equations is routine, while higher-order equations are more difficult to handle. Altogether, replacing gKdVB with gBBMB is a well-suited modelling choice; for an alternative justification using perturbative methods in the case $p=1$, see  \cite{MDLP2019}.
\par As referenced in the previous paragraph, modelling blood flow in relatively large subsets of the circulatory system demands we account for the influence of bifurcations (trifurcations, et cetera) in the arterial tree. Accordingly, the suitable formulation of gBBMB on a network, loosely understood for now to mean a collection of subintervals of $\mathbb{R}$ joined together at various points, has scientific merit. The study of blood flow models on networks began, to our knowledge, in $1986$ \cite{ZM1986} and remains of interest to biomedical engineers in modern times (see for example \cite{Alastruey_etal_2011}). Additionally, BBM has been studied on networks previously in \cite{BonaCascaval2008, MR2014, AC2019}.
\par By allowing the coefficients of gBBMB to vary between edges of the network in question we can also investigate how a flow is altered when it moves between two vessels with different elastic properties. Physiologically, the elasticity of a vessel can change due to arteriosclerosis. So, the transition between a healthy artery and an unhealthy one can be modelled by gBBMB on a network with two edges, with the coefficients $\mu, \nu, \gamma$ in \eqref{eqn:gbbmb} taking different values on either edge. Alternatively, the elastic properties of an unhealthy artery can be modified by inserting a small wire or polymer mesh called a stent. Studying gBBMB on a network with variable coefficients can therefore also help us understand how stents affect flow in the arterial tree. The analysis of blood flow in stented vessels has attracted considerable attention: \cite{Canic2002, CGLMTW2019, PMB2017, FPMM2019, FLQ2003} describe several different perspectives on analytical and numerical aspects of stent modelling.
\par Now, when the artery under consideration is very long, the use of a dispersive PDE such as gBBMB as a blood flow model is particularly well-justified: in order for dispersion to be an important influence on the movement of some material continuum, there must be enough room to let waves disperse. Thus we expect gBBMB to be a quality model of flow in the femoral artery, which is reasonably long on a physiological scale. So, we can model the effects of bifurcations, arteriosclerosis, and stents in the femoral artery (and its subarteries) by analyzing solutions to gBBMB on a network. The modelling of stents in the femoral artery is a problem of special interest to modern medical practitioners. Some clinical trials from $2017$ \cite{Nasr_etal_2017, Gouffeic_etal_2017} (see also the comments in \cite{DA2017}) indicate that stenting lesions in the ``common'' (upper) portion of the femoral artery can reduce post-operative complications in certain patients, when compared to surgical techniques like endarterectomy. According to a study from $2019$ \cite{Jia_etal_2019}, however, determining best practices for deciding whether or not to use a stent in the common femoral artery remains a complicated open question. While gBBMB is far too simple to provide solid quantitative answers to this medical problem, it may provide a good toy model for investigations into the basic physics of flow-stent interactions in the femoral artery. For example, studying gBBMB on a network may help determine whether or not stents in the femoral artery obstruct blood flow by reflecting incident pulses off the interface between stented and unstented regions. In the future, predictions made with gBBMB may be benchmarked against both simulations of more realistic model equations and actual patient data to discover if gBBMB indeed provides a useful reduced description of the relevant physics. 
\subsection{Outline of Paper and Relation to Previous Work}
The purpose of this paper is to provide a physically sound formulation of gBBMB on a family of simple networks, to establish the well-posedness of this formulation, and finally to exhibit the results of some numerical experiments based on a classical finite difference scheme. In section \ref{s:formulation}, we define star networks rigorously and justify our choice of compatibility conditions imposed at the junction; in particular, we shall see that the standard Kirchhoff conditions for PDEs on networks do not conserve mass for gBBMB in full generality. In section \ref{s:lwp}, we show how a routine fixed-point strategy can be used to prove local well-posedness for gBBMB. In section \ref{s:gwp}, we use an energy argument to extend local-in-time solutions to global ones (``global'' here means valid on the time interval $(0,\infty)$) in certain interesting special cases. In section \ref{s:num}, we present the results of some numerical simulations of gBBMB (based on \cite{EM1975,EM1977} and operator splitting) to understand how nonlinear waves scatter off the junction of a network with two edges. 
\par The presentation here is largely inspired by Bona and Cascaval's analysis of BBM on trees \cite{BonaCascaval2008}. Aside from our considering the more general gBBMB (with piecewise-constant coefficients) instead of BBM (with constant coefficients), there are a few differences between our approach and that of Bona and Cascaval:
\begin{enumerate}
\item{ Bona and Cascaval allow for networks with edges of finite or infinite length. We have chosen to ignore the finite-length case here, since the main focus of our theoretical and numerical investigations is the effect of the network's central junction on wave propagation.}
\item{In this paper, the compatibility conditions imposed at the junction guarantee that the ``mass" of our solution to gBBMB is conserved, regardless of the values the coefficients of the PDE take on each edge. Conversely, the classical Kirchhoff junction conditions imposed in \cite{BonaCascaval2008} only guarantee mass conservation when viscoelasticity is ignored \emph{and} only the dispersive term's coefficient $\mu$ is allowed to vary from edge to edge. The identification of such mass-conserving junction conditions is the main contribution of this article.}
\item{The fixed-point strategy for constructing local-in-time solutions here is essentially identical to the one proposed in \cite{BonaCascaval2008}. However, the extension of short-time solutions to solutions that exist for all time is not presented explicitly in \cite{BonaCascaval2008}, though it is indicated that \textit{a priori} bounds may be used to establish global-in-time well-posedness. Here, we compute the time derivative of a solution's energy and show that, while in many physically interesting cases this expression is enough to extend local solutions out to arbitrary times, it may fail to provide useful information even for BBM on a general star network.}
\end{enumerate}

\section{Formulation of gBBMB on a Star Network}\label{s:formulation}
\par First, we need to sensibly formulate gBBMB on a class of spatial domains that may include graph-like bifurcations. Towards this goal, we rigorously define the concept of a star network:
\begin{defn}\label{defn:network}
\begin{enumerate} 
\
\item{Let $\left\{e_{i}\right\}_{i=1}^{N}$ be a collection of subintervals of $\mathbb{R}$, such that each $e_{i}$ is either $(-\infty,0]$ or $[0,\infty)$.The \textbf{star network} $X$ associated to this data is the disjoint union of all the intervals $e_{i}$ modulo identifying all $0$'s to a single point.}
\item{The $e_{i}$'s appearing above are called the \textbf{edges} of the network, and the equivalence class of (any) $0\in e_{i}$ is called the \textbf{junction} of the network. }
\item{Any edge that is a copy of $(-\infty,0]$ is said to be \textbf{incoming}, and any edge that is a copy of $[0,\infty)$ is said to be \textbf{outgoing}.}
\end{enumerate}
\end{defn}

\par Definition \ref{defn:network} serves as our model for the femoral artery and its subarteries. 
\par Notice that our definition accounts for neither the curvature of the blood vessels nor the angles between vessels meeting at a junction, so important physics is likely being ignored. Figure \ref{fig:network} illustrates how our definition ignores the physically relevant embedding of a network into Euclidean space. Quantifying the effects of edge curvature and angles between edges, perhaps following the ``limiting'' approach of \cite{NS2015}, could make for interesting future work.
\par Intuitively, the solution to a PDE on a network ought to be viewed as a global object, but since we have no tools to define spatial derivatives at the junction, differential operators  lack an obvious global interpretation. Of course, since our networks are built from subintervals of $\mathbb{R}$, there is no problem in understanding these differential operators on each edge. Thus, when we speak of a PDE on a star network we really mean a system of PDEs defined on each edge $e_{i}$, coupled together by conditions imposed at the junction. With all this in mind, our formulation of gBBMB on a network is thus: let $X$ be a star network with edges $e_{i} \ (i=1,...,N)$ and suppose we are given coefficients 
$$\mu_{i}, \alpha_{i}, \nu_{i}\geq0, \ \gamma_{i}\in [0,1] \quad \forall \ i,$$ then we seek functions $u_{i}(x,t)$ defined for $ (x,t)\in e_{i}\times [0,\infty)$ (with suitable regularity) satisfying the system
\begin{subequations}\label{eqn:rough_gen_sys}
\begin{align}
\hspace{-1cm} \left(1-\mu_{i}^2\partial_{x}^2\right)u_{i,t} +\partial_{x}\left(\alpha_{i} u_{i} + \frac{\gamma_{i}}{p+1}u_{i}^{p+1}\right) -\nu_{i} u_{i,xx} &=0 \quad \text{on} \ (x,t)\in \mathrm{Int}\left(e_{i}\right)\times [0,\infty) \ \forall \ i \\
&+ \ \text{conditions at junction}
\\
&+ \ \text{initial conditions}. 
\end{align}
\end{subequations}
Note that allowing the coefficients of gBBMB to vary from edge to edge is a critical step in modelling the effects of arteriosclerosis or vascular stents: at the interface between a healthy artery and sclerotic or stented artery, the elastic properties of a blood vessel may change \cite{TMMM2008}. 
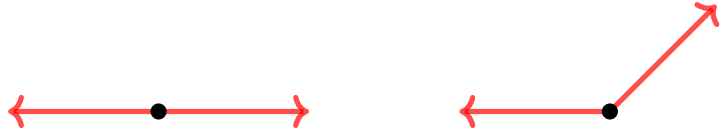
\begin{figure}
\centering
\begin{tikzpicture}
\draw [->, red, line width=2, opacity=0.7] (2,0) -- (0,0);
\draw [->, red, line width=2, opacity=0.7] (2,0) -- (4,0);
\draw [fill] (2,0) circle [radius=0.1];

\draw [->, red, line width=2, opacity=0.7] (8,0) -- (6,0);
\draw [->, red, line width=2, opacity=0.7] (8,0) --(8+1.41421356237,1.41421356237);
\draw [fill] (8,0) circle [radius=0.1];
\end{tikzpicture}
\caption{Two subsets of $\mathbb{R}^2$ corresponding to the same network by the definition used here. The junctions are denoted by black dots, and the arrowheads indicate that the edges each have infinite length.}
\label{fig:network}
\end{figure}
\par We now determine physically appropriate junction conditions for gBBMB.  Let $X$ be a star network with edges $e_{i} \ (i=1,...,N)$, and assume we are given $N$ functions $u_{i}(x,t)$, each solving gBBMB on $e_{i}$. Denote the junction of $X$ by $0$. We need $N$ conditions, one per $u_{i}$, if we want any hope of well-posedness. Continuity, in the sense that for all $i,j$ we have
\begin{equation} \label{eqn:cts_soln}
u_{i}(0,t)= u_{j}(0,t),
\end{equation}
is an obvious and physically well-motivated constraint for gBBMB: we expect the nonlinearity (the $\gamma$-term) in the PDE to be mollified by dispersion (the $\mu$-term) and dissipation (the $\nu$-term), so shock formation appears to be unlikely. Notice that, if our solution $(u_{1},...,u_{N})$ is continuous in the above sense, it defines a continuous function $u(x,t)$ on $X\times[0,\infty)$ satisfying
$$ u(x,t)|_{e_{i}} = u_{i}(x,t) . $$
However, continuity only yields $n-1$ equations, so we need one more constraint. 
\par One common choice to close the system of junction constraints is the Kirchhoff condition, which says that
\begin{equation}\label{eqn:KC}
\sum \mu_{\text{in}}^2 u_{\text{in},x}(0,t) = \sum \mu_{\text{out}}^2 u_{\text{out},x}(0,t) \quad \forall \ t,
\end{equation}
where the subscripts ``in'' and ``out'' represent values on edges where a signal is leaving or arriving, respectively.  The Kirchhoff condition is very popular in the field of PDEs on networks \cite{Mugnolo2014}, and has been applied to BBM on networks in \cite{BonaCascaval2008, MR2014}. According to \cite{BonaCascaval2008}, BBM is well-posed for short times on trees subject to  the demand of a continuous solution and the Kirchhoff condition, as well as regularity requirements on the initial and boundary conditions.
\par In this investigation, however, we use an alternative junction condition that guarantees mass conservation even when viscoelasticity is present and all coefficients are allowed to vary from edge to edge. We call a function $g(x,t, u, u_{i,x}, u_{i,t},...)$ ``globally conserved'' under the evolution of gBBMB if 
$$ \frac{\diff}{\diff t} \sum_{i} \int_{e_{i}} g(x,t, u, u_{i,x}, u_{i,t},...) \ \diff x = 0. $$ 
Let
\begin{equation}
f_{i}(u_{i}) = \alpha_{i} u_{i} + \frac{\gamma_{i}}{p+1} u_{i}^{p+1}
\end{equation}
denote the advective flux on the edge $e_{i}$. Then, by inspection, the solution $u(x,t)$ to gBBMB on $X$ is globally conserved if and only if
\begin{equation}\label{eqn:MC}
\hspace{-1cm}
\sum \left[-\mu_{\text{in}}^2 u_{\text{in},xt} +f_{\text{in}}(u_{\text{in}}) - \nu_{\text{in}} u_{\text{in}, x}\right]_{x=0} =\sum \left[-\mu_{\text{out}}^2 u_{\text{out},xt} +f_{\text{out}}(u_{\text{out}}) - \nu_{\text{out}} u_{\text{out}, x}\right]_{x=0} ,
\end{equation}
where the notation is the same as that used in \eqref{eqn:KC}. Since $$\sum_{i}\int_{e_{i}}u(x,t) \ \diff x$$ physically represents the volume bounded by a network of elastic blood vessels and we assume the blood conducted by our artery has constant density, we can justifiably call \eqref{eqn:MC} the ``mass conservation condition''. This condition tells us that the amount of fluid contained in our system remains constant for all time, a critical constraint to impose from a physical perspective. In the literature, the Kirchhoff condition is sometimes considered equivalent to mass conservation  for BBM. If the $\alpha_{i}$'s and $\gamma_{i}$'s are constant throughout the network and $\nu_{i}=0 \ \forall \ i$ then the Kirchhoff condition does imply the mass conservation. However, the Kirchhoff condition does not even guarantee \emph{edge-wise} mass conservation for general coefficients. In the remainder of this work, therefore, we close the gBBMB system on a network by imposing continuity and the mass conservation condition, rather than the Kirchhoff condition.
\section{Local Well-Posedness}\label{s:lwp}
Let $X$ be a star network with edges $e_{i} \ (i=1,...,N)$. In this section, we prove that gBBMB is locally-in-time well-posed on $X$, subject to the continuity and mass conservation junction conditions.  We view $X$ as having a single incoming edge $e_{1}$ from which signals arrive at the junction, and all other edges are outgoing. This covers two especially significant special cases: 
\begin{itemize}
\item{the case where $X$ has two edges, physically representing a healthy blood vessel sharply transitioning into a sclerotic or stented blood vessel (or vice versa);}
\item{the ``Y-network'' $\Y$ depicted in Figure \ref{fig:y_junction}, where we view the edge $e_{1}$ as a copy of $(-\infty,0]$ and the edges $e_{2}, e_{3}$ as copies of $[0,\infty)$. Such a network serves as a simple model of a bifurcation in the femoral artery.}
\end{itemize}
Only small modifications are required to handle any number of incoming edges, so we ignore such a general setup here for the sake of conceptual clarity. 
\par The forthcoming analysis closely follows the methodology of \cite{BonaCascaval2008} and its antecedents, though as stressed  in the previous section the choice of junction conditions is novel, and applicable in more general circumstances.  
\begin{figure}
\centering
\begin{tikzpicture}[scale=1.65]
\draw [->, gray, line width=3, opacity=1.] (8,0) -- (6,0);
\node[gray] at (7+0.1,0.2) {$e_{1}$};
\draw [->,blue, line width=3, opacity=1.] (8,0) --(8+1.41421356237,1.41421356237);
\node[blue, align=left] at (8+1.41421356237/2-0.3,1.41421356237/2+0.15) {$e_{2}$};
\draw [->,red, line width=3, opacity=1.] (8,0) --(8+1.41421356237,-1.41421356237);
\node[red, align=right] at (8+1.41421356237/2+0.2,-1.41421356237/2+0.1){$e_{3}$};
\draw [fill] (8,0) [radius=0.1] circle;
\end{tikzpicture}
\caption{Diagram of the Y-network $\Y$ introduced at the beginning of Section \ref{s:lwp}. The arrowheads indicate that the edges each have infinite length.} 
\label{fig:y_junction}
\end{figure}
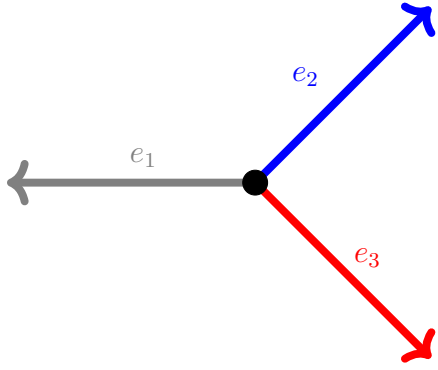
\subsection{Function Spaces}
In this subsection, we briefly review the function spaces required to formulate gBBMB on $X$. All spaces defined here are either well-known in analysis, or have been introduced previously in \cite{BB1973, BonaLuo1995,  BonaCascaval2008}. Let $U\subseteq\RN$ be open and let $T>0$ (we allow $T=+\infty$). Additionally, let $|\cdot|\colon \RN\rightarrow \mathbb{R}$ denote the Euclidean norm. 
\begin{itemize}
 \item{$C^{k}_{b}\left(U\right)$ denotes the space of functions on $U$ whose derivatives up to order $k$ are continuous and bounded; this becomes a Banach space when endowed with the norm 
 \begin{equation}\label{eqn:Ck_norm}
 \| u\|_{C^{k}_{b}\left(U\right)} \doteq \sup_{|\alpha|\leq k}\sup_{U} \left| D^{\alpha}u \right|,
\end{equation}
where $\alpha$ is a multi-index and $D^{\alpha} =\partial^{\alpha_{1}}_{x^{1}}\partial^{\alpha_{2}}_{x^{2}}\cdots$;
 }
 \item{$C_{b}\left( U \right)\doteq C^{0}_{b}\left(U \right)$;}
\item{$L^{p}\left(U\right)$ denotes the Banach space of real-valued functions on $U$ whose absolute values have integrable $p^{\text{th}}$ powers;}
 \item{$H^{k}\left(U\right)$ denotes the Hilbert space of square-integrable, $k$-times weakly differentiable functions on $U$ with square-integrable derivatives;}
 \item{given a Banach space $A$, $C(0,T;A)$ denotes the Banach space of all continuous functions $u\colon[0,T]\rightarrow A$ equipped with the norm 
 \begin{equation}
 \| u \|_{C(0,T;A)} = \sup_{[0,T]} \|u(t)\|_{A}.
 \end{equation}}
  \end{itemize}
Of course, we can also define $C_{b}(U)$ for a not necessarily open set $U\subseteq\RN$, and in this situation $C_{b}\left(U\right)$ remains a Banach space with respect to the sup-norm. 
\par For the analysis of gBBMB on the spatial domain $[0,\infty)$, some other function spaces have appeared in the literature, and we define these now:
\begin{defn}
\begin{align*}
 \mathcal{B}^{k,\ell}_{T} &\doteq \left\{
 u \in C_{b}\left([0,\infty)\times [0,T]\right)
  \ | \ \text{for all} \ 0\leq i\leq k, \ 0\leq j \leq \ell,
   \ \partial_{t}^{k}\partial_{x}^{\ell}u \in C_{b}\left([0,\infty)\times [0,T] \right)\right\}, \quad \text{and}
 \\ 
 \mathcal{B}_{T}&\doteq \mathcal{B}^{0,0}_{T}. 
\end{align*}
\end{defn}
 Finally, we define function spaces specific to the spatial domain $X$:
\begin{defn} 
\begin{align*}
 C_{b}\left(X\right) &\doteq \{(u_{1},...,u_{N})\in \left(C_{b}[0,\infty)\right)^{N}\ | \ u_{1}(0)=\cdots=u_{N}(0)  \}, 
 \\
 L^{2}\left(X\right)  &\doteq \left(L^{2}(0,\infty)\right)^{N}, \quad \text{and}
\\
 H^{1}\left(X\right)  &\doteq \left(H^{1}(0,\infty)\right)^{N} \cap C_{b}\left(X\right) . 
 \end{align*}
 An element of $\cbx$ can be identified with a bounded, continuous function $u\colon X\rightarrow\mathbb{R}$ defined by 
 $$ u|_{e_{i}} \doteq u_{i}.$$
 In this notation, $\cbx$ becomes a Banach space when equipped with the norm
 \begin{equation}
 \| u \|_{\cbx} \doteq \max_{i=1,2,3} \| u_{i}\|_{C[0,\infty)}. 
 \end{equation}
Additionally, $L^{2}\left(X\right),\  H^{1}\left(X\right)$ become Hilbert spaces when equipped with the sum inner product. 
\end{defn}
Note that the Sobolev embedding $H^{1}\left(0,\infty\right)\subseteq C_{b}\left(0,\infty\right) $ has been tacitly used in the above definition: to define $H^{1}\left(X\right)$,  we must know the values of an element of $H^{1}(0,\infty)$ at $x=0$. For a more complete discussion on this point, see \cite{BB1973}*{Proposition $1$}. Additionally, the same Sobolev embedding gives us the inclusion $H^{1}\left(X\right)\subseteq C_{b}\left(X\right) $. 
\subsection{Review of Fixed-Point Formulation of gBBMB on a Half-Line}
In this subsection, we review the ideas behind the proof of local well-posedness of gBBMB posed on $(x,t)\in [0,\infty)^2$. The idea is to express the solution to gBBMB as the fixed point of a certain nonlinear integral operator on $\mathcal{B}_{T}$ for small enough $T>0$. Throughout, we denote the advective flux in gBBMB by 
\begin{equation}
f(u) = \alpha u +\frac{\gamma}{p+1} u^{p+1}. 
\end{equation}
We are interested in solving the following problem: for given $h(t)\in C[0,\infty), \ \varphi(x)\in C_{b}[0,\infty)$, find $T>0$ and $u(x,t) \in \mathcal{B}^{1,2}_{T}$ such that 
\begin{subequations}
\label{eqn:gbbmb_qp}
\begin{align}
(1-\mu^2\partial_{x}^2)u_{t} + \left(f(u)\right)_{x} - \nu u_{xx} &=0 \quad \forall \ (x,t) \in (0,\infty)^2,
\\ 
u(0,t) &= h(t) \quad \forall \ t\in [0,\infty),
\\ 
u(x,0) &= \varphi(x) \quad \forall \ x\in [0,\infty). 
\end{align}
\end{subequations}
To recast the above system as a fixed point problem, we use the Green's function of $1-\mu^{2}\partial_{x}^2$:
\begin{lemma}\label{lemma:green_fnc} Let $\delta(x-y)$ denote the Dirac function centred at $y\in \mathbb{R}$. 
The function
\begin{equation} 
G(x,y) \doteq -\frac{1}{2\mu} \left(e^{\frac{-(x+y)}{\mu}}- e^{-\frac{|x-y|}{\mu}}\right)\colon [0,\infty)^2\rightarrow \mathbb{R}
\end{equation}
satisfies the PDE
\begin{equation}
\left(1-\mu^2\partial_{x}^2\right)G(x,y) =\delta(x-y) \quad \forall \ (x,y)\in (0,\infty)^2
\end{equation}
in the sense of distributions, with $G(0,y)=0$ and $\lim_{x\rightarrow \infty} G(x,y)=0 \ \forall \ y \in [0,\infty).$  
\end{lemma}
\qed \\
In light of Lemma \ref{lemma:green_fnc}, we may rewrite the PDE in \eqref{eqn:gbbmb_qp} as 
\begin{align} \label{eqn:ut}
u_{t} = h'(t)e^{-\frac{x}{\mu}} + \int_{0}^{\infty} G(x,y) \ \left(\left(-f(u(y,s))\right)_{y} + \nu u_{yy}(y,s)\right) \ \diff y.
\end{align}Integrating by parts and solving a linear first order ODE in time, we arrive at the fixed point problem \cite{BonaLuo1995}*{Equations $3.3$-$3.8$}
\begin{subequations}
\begin{align}
K(x,y) &\doteq \frac{1}{2\mu^2}  \left(e^{\frac{-(x+y)}{\mu}} +\sgn(x-y) \ e^{-\frac{|x-y|}{\mu}}\right),
\\ 
\mathbb{B}_{\text{adv}}[u](x,t) &\doteq  \int_{0}^{t}\int_{0}^{\infty} e^{-\frac{\nu}{\mu^2}(t-s)} \ K(x,y) \ f(u(y,s)) \ \diff y \ \diff s,
\\ 
\mathbb{B}_{\text{visc}}[u](x,t) &\doteq  \frac{\nu}{\mu^2} \int_{0}^{t}\int_{0}^{\infty} e^{-\frac{\nu}{\mu^2}(t-s)} \ G(x,y) \ u(y,s) \ \diff y \ \diff s,
\\ 
u(x,t) &= e^{-\frac{\nu t}{\mu^2}}\varphi(x) + \left(h(t)-h(0)e^{-\frac{\nu t}{\mu^2}}\right) e^{-\frac{x}{\mu}}  \label{eqn:u_fix_op}
\\ &\phantom{=}+\mathbb{B}_{\text{adv}}[u](x,t) + \mathbb{B}_{\text{visc}}[u](x,t). \nonumber
\end{align}
\end{subequations}
We have chosen the notation $\mathbb{B}_{\text{adv}}$ and $\mathbb{B}_{\text{visc}}$ because, when viscoelasticity is ignored, $\nu=0$ and $\mathbb{B}_{\text{visc}}[u]\equiv 0$. Additionally, $\mathbb{B}_{\text{adv}}$ contains all information on how advection affects the dynamics.  
\par If we treat the right-hand side of \eqref{eqn:u_fix_op} as a nonlinear operator on $\mathcal{B}_{T}$, we can use the contraction mapping theorem to argue that a solution $u$ to \eqref{eqn:gbbmb_qp} exists for some time $T$, provided $h$ and $\varphi$ admit enough derivatives \cite{BonaLuo1995}*{Proposition $3.1$, Lemma $3.3$}. The particular form of the nonlinear integral operator appearing in the proof allows one to conclude that the fixed point is actually in $\mathcal{B}^{1,2}_{T}$, and therefore the fixed point is a classical solution of gBBMB. 
\par We conclude by stating the following lemma, which helps some calculations in the next subsection: 
\begin{lemma}\label{lemma:xt_der}
 If $u(x,t)$ solves \eqref{eqn:gbbmb_qp}, then
\begin{align*}
f(h(t)) - \left[\mu^2 u_{xt} + \nu u_{x}\right]_{x=0} = \mu h'(t) +  \frac{\nu}{\mu} h(t) - \frac{1}{\mu}\int_{0}^{\infty} e^{\frac{-y}{\mu}} \left[f(u(y,t))-\frac{\nu}{\mu}u(y,t)\right] \ \diff y. 
\end{align*}
\end{lemma}
\begin{proof}
Differentiate both sides of \eqref{eqn:ut} with respect to $x$, then integrate by parts to get rid of all the derivatives in the integrand. 
\end{proof}
\subsection{Fixed-Point Formulation of gBBMB on $X$}
We now adapt the techniques from the previous subsection to prove local well-posedness of gBBMB on a star network $X$ with $N$ infinitely long edges $e_{i}$. Recall that we want to focus on the case of one incoming edge $e_{1}$, from which a signal arrives at the junction and scatters off into the other edges. Let
\begin{equation}
f_{i}(u_{i}) = \alpha_{i} u_{i} + \frac{\gamma_{i}}{p+1} u_{i}^{p+1}
\end{equation}
denote the advective flux on the edge $e_{i}$. For $\varphi\in C_{b}\left(X\right)$ with $\varphi_{i}\doteq \varphi|_{e_{i}}$, our formulation of gBBMB on $X$ then reads
\begin{subequations}\label{eqn:nl_y_junction_sys}
\begin{align}
\left(1-\mu_{1}^2\partial_{x}^2\right)u_{1,t} +\left(f_{1}(u_{1})\right)_{x} -\nu_{1}u_{1,xx}&=0 \quad \text{on} \ (x,t)\in(-\infty,0)\times (0,\infty),
 \\
\left(1-\mu_{i}^2\partial_{x}^2\right)u_{i,t} +\left(f_{i}(u_{i})\right)_{x} -\nu_{i}u_{i,xx} &=0 \quad \text{on} \ (x,t)\in  (0,\infty)^2, \  i=2,..., N, 
\\
\mu_{1}^2 \ u_{1,xt}(0,t) -f_{1}(u_{1}(0,t)) +\nu_{1} u_{1,x}(0,t) &= \sum_{i= 2}^{N} \mu_{i}^2 \ u_{i,xt}(0,t) \nonumber
\\ 
&\phantom{=} -f_{i}(u_{i}(0,t)) +\nu_{i} u_{i,x}(0,t)  \quad \forall \ t\in [0,\infty), 
 \label{eqn:nl_y_junction_mass}
\\ 
u_{i}(0,t)&=u_{j}(0,t) \quad \forall \ i, j =1,...,N, \ t\in [0,\infty), \label{eqn:nl_y_junction_cts} \\ 
u_{i}(x,0) &= \varphi_{i}(x) \quad \forall \ x\in e_{i}, \ i=1,...,N. 
\end{align}
\end{subequations}
Assuming a classical solution $(u_{1},u_{2},..., u_{N})$ to \eqref{eqn:nl_y_junction_sys} exists for some time $T$, let $$u(x,t)\in C\left(0,T, C_{b}\left(X\right)\right)$$ be defined by
$$ u(x,t)|_{e_{i}}= u_{i}(x,t). $$

\par We attack this problem by casting \eqref{eqn:nl_y_junction_sys} as a fixed-point problem on $C\left(0,T, C_{b}\left(X\right)\right)$. To do this, we write out the integral form of gBBMB on each $e_{i}$ in terms of the \textit{a priori} unknown common junction value 
$$ h(t) \doteq u_{1}(0,t)=\cdots= u_{N}(0,t). $$ 
Then, we use the mass conservation condition to write out a linear initial-value problem for $h(t)$ with $u$-dependent forcing, which is trivially solvable in terms of $u$. 
\par As in \cite{BonaCascaval2008}, we start by changing variables $x\mapsto -x$ in $e_{1}$ to make sure all $u_{i}$'s are defined on the same spatial domain $[0,\infty)$. Letting $\sigma_{i}=-1$ if $i=1$ and $\sigma_{i}=1$ otherwise, the integral form of gBBMB on $e_{i}$ can be written as 
\begin{align}
u_{i}(x,t) &= e^{-\frac{\nu_{i} t}{\mu_{i}^2}}\varphi_{i}(x) + \left(h(t)-h(0)e^{-\frac{\nu_{i} t}{\mu_{i}^2}}\right) e^{-\frac{x}{\mu_{i}}} +\sigma_{i}  \mathbb{B}_{\text{adv},i}[u_{i}](x,t) + \mathbb{B}_{\text{visc},i}[u_{i}](x,t).
\end{align}
Adapting Lemma \ref{lemma:xt_der} gives 
\begin{align}\label{eqn:xt_der_i}
\hspace{-1cm}
 \sigma_{i} f_{i}(h(t))-\left[\mu_{i}^2 u_{i,xt}+ \nu_{i} u_{i,x}\right]_{x=0} = \mu_{i} h'(t) +  \frac{\nu_{i}}{\mu_{i}} h(t) - \frac{1}{\mu_{i}}\int_{0}^{\infty} e^{\frac{-y}{\mu_{i}}} \left[\sigma_{i}f_{i}(u_{i}(y,t))-\frac{\nu_{i}}{\mu_{i}}u_{i}(y,t)\right] \ \diff y. 
\end{align}
Now, after changing variables, we can write \eqref{eqn:nl_y_junction_mass} as
\begin{equation} \label{eqn:mass_cons_Y}
\sum_{i} \sigma_{i} f_{i}(h(t))- \left[\mu_{i}^2 u_{i,xt} + \nu_{i} u_{i,x}\right]_{x=0} =0. 
\end{equation}
Combining this with \eqref{eqn:xt_der_i} and defining $\mu_{*}\doteq \sum_{i}\mu_{i}, \ \nu_{*}\doteq\sum_{i}\frac{\nu_{i}}{\mu_{i}}$, we get 
\begin{align}\label{eqn:h_ode_full}
h'(t) +\frac{\nu_{*}}{\mu_{*}} h(t) &= \sum_{i} \frac{1}{\mu_{i}\mu_{*} } \int_{0}^{\infty} e^{\frac{-y}{\mu_{i}}} \left[\sigma_{i}f_{i}(u_{i}(y,t))-\frac{\nu_{i}}{\mu_{i}}u_{i}(y,t)\right] \ \diff y.
\end{align}
Finding the junction value $h(t)$ thus amounts to solving a linear, parameterized (by $u$) ODE \eqref{eqn:h_ode_full} subject to the initial condition $h(0)=\varphi(0)$. This is trivial, however: 
\begin{equation}\label{eqn:h_expl}
h(t)=\varphi(0)e^{-\frac{\nu_{*}t}{\mu_{*}}} + \sum_{i} \frac{1}{\mu_{i}\mu_{*}} \int_{0}^{t} \int_{0}^{\infty} e^{-\left(\frac{\nu^{*}}{\mu_{*}}(t-s) + \frac{y}{\mu_{i}}\right)} \left[\sigma_{i}f_{i}(u_{i}(y,s))-\frac{\nu_{i}}{\mu_{i}}u_{i}(y,s)\right] \ \diff y \ \diff s. 
\end{equation}
Let 
\begin{equation}
\Phi[u] \doteq h(t)-\varphi(0),  
\end{equation}
with $h(t)$ given by \eqref{eqn:h_expl}. Then, we may write the fixed-point formulation of gBBMB on $X$ as follows: find $u_{i} \ (i=1,...,N)$ such that 
\begin{equation}\label{eqn:fixed_y_junc_i}
\hspace{-1cm}
u_{i}(x,t) = e^{-\frac{\nu_{i} t}{\mu_{i}^2}}\varphi_{i}(x) + \left(\Phi[u]+\varphi(0)\left(1-e^{-\frac{\nu_{i} t}{\mu_{i}^2}}\right)\right) e^{-\frac{x}{\mu_{i}}} +\sigma_{i}  \mathbb{B}_{\text{adv},i}[u_{i}](x,t) + \mathbb{B}_{\text{visc},i}[u_{i}](x,t). 
\end{equation}
Notice how coupling between individual edges is described entirely by $\Phi[u]$. 
\par Now, we are at last ready to state and prove our main theorem for this section.  
\begin{thm}\label{thm:gbbmb_lwp} Given $\varphi\in C_{b}\left(X\right)$ with $\varphi_{i}\in C^{2}_{b}(0,\infty)$ for each $i$, there exists a $T>0$ and a unique $u\in C\left(0,T; C_{b}\left(X\right)\right)$ such that $u(x,t)$ is a classical solution to gBBMB on $X$ satisfying the mass conservation condition \eqref{eqn:MC}. Further, $u|_{e_{i}}\in \mathcal{B}^{1,2}_{T} \ \forall \ i$ and $u$ depends continuously on the initial data $\varphi$. 
\end{thm}
\begin{proof} (Sketch)
We begin by choosing any $T>0$. For brevity, let us define $$A\doteq C\left(0,T; C_{b}\left(X\right)\right).$$ We then pick any $R>0$ and define 
\begin{equation}\label{eqn:B_defn}
B\doteq B(0,R) \subseteq A;
\end{equation}
correct choices of $R$ and $T$ emerge naturally in the course of the proof. Let $\Psi\colon A\rightarrow A$ be defined by 
\begin{equation}
\hspace{-1cm}
\Psi[u]|_{e_{i}} \doteq e^{-\frac{\nu_{i} t}{\mu_{i}^2}}\varphi_{i}(x) + \left(\Phi[u]+\varphi(0)\left(1-e^{-\frac{\nu_{i} t}{\mu_{i}^2}}\right)\right) e^{-\frac{x}{\mu_{i}}} +\sigma_{i}  \mathbb{B}_{\text{adv},i}[u_{i}](x,t) + \mathbb{B}_{\text{visc},i}[u_{i}](x,t). 
\end{equation}
We construct our local solution as a fixed point of $\Psi$. $\Psi$ maps $B$ to itself provided 
\begin{equation}\label{eqn:B_to_B_condition}
\left\|\varphi\right\|_{\cbx} + TR\left( 1+ R^{p}\right) \leq c_{1}R
\end{equation}
for some constant $c_{1}$ depending only on the parameters $\mu_{i}, \alpha_{i}, \gamma_{i}$, and  $\nu_{i}$. 
Further, $\Psi$ is a contraction mapping if
\begin{equation}\label{eqn:contr_condition}
T \ \left( 1 + R^{p}\right) < c_{2} 
\end{equation}
where $c_{2}$ is a constant depending on $\mu_{i}, \alpha_{i}, \gamma_{i}$, and  $\nu_{i}$. The two constraints \eqref{eqn:B_to_B_condition} and \eqref{eqn:contr_condition} are satisfied if we choose
\begin{subequations}\label{eqn:fp_require}
\begin{align}
R &\geq \max\left\{1, \frac{2}{c_{1}} \right\} \ \| \varphi \|_{\cbx}\quad \text{and}
\\ 
T &<  \frac{\min\left\{ \frac{c_{1}}{2}, c_{2}\right\}}{1+R^{p}}.
\end{align}
\end{subequations}
Now, apply the contraction mapping theorem to see that $\Psi$ has a unique fixed point in $B$ if $R$ and $T$ satisfy \eqref{eqn:fp_require}. Unconditional uniqueness of the fixed point can be established by a routine bootstrap argument. From the definition of $\Psi$, this fixed point depends continuously on $\varphi$. Due to the nested integrals in the definition of $\Psi$, the claimed regularity of the fixed point given a smooth enough $\varphi$ is also obvious. We conclude that the fixed point is actually a classical solution to gBBMB. 
\end{proof}
\section{Global Well-Posedness via Energy Methods}\label{s:gwp}
 We now determine if and when our local-in-time solution may be extended to exist for an arbitrary time. Specifically, we show that, given sufficiently regular initial data, the local-in-time solution $u(x,t)$ to \eqref{eqn:nl_y_junction_sys} obtained from  Theorem \ref{thm:gbbmb_lwp} satisfies $u(\cdot, t)\in H^{1}\left(X\right) \ \forall \ t \in [0,T]$. Following this, we compute the time evolution of the energy (squared $H^{1}\left(X\right)$-norm) of our solution, which in turn allows us to prove well-posedness for \eqref{eqn:nl_y_junction_sys} in $C_{b}\left([0,T), H^{1}\left(X\right)\right)$ for \emph{any} $T>0$ in the following two cases: 
\begin{itemize}
\item{$p$ is even and certain physically relevant restrictions are imposed on the coefficients $\alpha_{i}$ and $\gamma_{i}$;}
\item{ $\sum \sigma_{i}\alpha_{i} = \sum \sigma_{i}\gamma_{i}=0$.}
\end{itemize}

\par Throughout this subsection, let $X$ be a star network with edges $e_{i} \ (i=1,...,N)$. We denote our initial data by $\varphi\in C_{b}\left(X\right)$ with $\varphi_{i}\in C^{2}_{b}[0,\infty)$ for each $i$. Finally, we let $u(x,t)$ denote the classical solution to gBBMB on $X$ valid up to time $T= \mathcal{O}\left( \left(1+\| \varphi\|_{C_{b}\left(X\right)}\right)^{-p}\right)$ whose existence is guaranteed by Theorem \ref{thm:gbbmb_lwp}. Additionally, we remind the reader that Sobolev embedding yields $H^{1}\left(X\right)\subseteq C_{b}\left(X\right)$. 
\par First, we need a helpful lemma characterizing the far-field behaviour of solutions to  \eqref{eqn:nl_y_junction_sys}. This result can be obtained by adapting the proof of Lemma $3$ in \cite{BB1973}:
\begin{lemma}\label{lemma:nullity}
Assume that all $\varphi_{i}$'s and their derivatives converge to $0$ as $x\rightarrow \infty$. Then, the functions $u_{i}(x,t)$ and all their derivatives converge to $0$ as $x\rightarrow \infty$, uniformly in $t$. 
\end{lemma}
\qed \\ 
We can now begin studying the $H^{1}$ theory of \eqref{eqn:nl_y_junction_sys}. Note that we use a model-dependent energy norm that is equivalent to the usual $H^{1}$ norm. 
\begin{defn} The \textbf{energy} $E(t)$ of the solution $u(x,t)$ to \eqref{eqn:nl_y_junction_sys}  is defined by
\begin{equation}
E(t) \doteq \frac12 \| u \|_{H^{1}\left(X\right)}^2 = \frac12 \sum_{i} \int_{0}^{\infty} |u_{i}|^2 + \mu_{i}^2 \ |u_{i,x}|^2 \ \diff x, 
 \end{equation}
provided all of the integrals are finite. 
\end{defn}
Next, we exhibit conditions under which $u(x,t)$ lies in $H^{1}\left(X\right)$ for all time and determine the evolution of $u$'s energy. 
\begin{thm}\label{thm:energy_X} If $\varphi\in H^{1}\left(X\right) \cap \left(C^{2}_{b}[0,\infty)\right)^{N}$, then $u(\cdot, t)\in H^{1}\left(X\right) \ \forall \ t \in [0,T]$. Further, for such $\varphi$, we have that the energy of $u(x,t)$ satisfies
\begin{equation}\label{eqn:energy_ode}
\frac{\diff E}{\diff t} = -h^{2}(t)\ \left[\sum_{i}\sigma_{i}\left(\frac{\alpha_{i}}{2}+ \frac{\gamma_{i}}{(p+1)(p+2)}h^{p}(t)\right)\right] -  \sum_{i} \nu_{i} \int_{0}^{\infty} |u_{i,x}|^2 \ \diff x,
\end{equation}
where as above $h(t)=u_{1}(0,t)=\cdots =u_{N}(0,t)$. 
\end{thm}
\begin{proof}
We follow the proof of \cite{BB1973}*{Lemma $4$}. Pick any $L>0$, then multiply both sides of gBBMB on each edge by $2u_{i}$ and integrate with respect to $x$ over $[0,L]$ to see that 
\begin{align*}
0 &= \int_{0}^{L}  \partial_{t}\left(|u_{i}|^2+\mu_{i}^2|u_{i,x}|^2\right) + 2\sigma_{i}\ \partial_{x}\left(\frac{\alpha_{i}}{2} u_{i}^2 + \frac{\gamma_{i}}{p+2} u_{i}^{p+2}\right) +2\nu_{i} \ |u_{i,x}|^2 \ \diff x -2\left[\mu_{i}^2 u_{i}u_{i,xt}+\nu_{i}u_{i}u_{i,x}\right]_{0}^{L}
\\ &= \int_{0}^{L}  \partial_{t}\left(|u_{i}|^2+\mu_{i}^2|u_{i,x}|^2\right) +2\nu_{i} \ |u_{i,x}|^2 \ \diff x + 2\left[\sigma_{i}\left(\frac{\alpha_{i}}{2} u_{i}^2 + \frac{\gamma_{i}}{p+2} u_{i}^{p+2}\right) - u_{i}\left( \mu_{i}^2 u_{i,xt}+\nu_{i}u_{i,x}\right)\right]_{0}^{L}. 
\end{align*}
Adding up the above expressions for $i=1,...,N$ and using the mass conservation junction condition \eqref{eqn:mass_cons_Y}, we obtain 
\begin{align}
\frac12 \sum_{i} \int_{0}^{L}  \partial_{t}\left(|u_{i}|^2+\mu_{i}^2|u_{i,x}|^2\right) \ \diff x \nonumber
&= h^{2}(t)\ \left[\sum_{i}\sigma_{i}\left(\frac{\alpha_{i}}{2}+ \frac{\gamma_{i}}{p+2}h^{p}(t)\right)\right]
\\ \nonumber &\phantom{=} -h(t) \ \left[\sum_{i}\sigma_{i}\left(\alpha_{i}h(t)+ \frac{\gamma_{i}}{p+1}h^{p+1}(t)\right)\right]   - \sum_{i}\int_{0}^{L}  \nu_{i} \ |u_{i,x}|^2 \ \diff x 
\\  \nonumber &\phantom{=} -\sum_{i}  \left\{\sigma_{i}\left(\frac{\alpha_{i}}{2} u_{i}^2 + \frac{\gamma_{i}}{p+2} u_{i}^{p+2}\right)- u_{i}\left( \mu_{i}^2 u_{i,xt}+\nu_{i}u_{i,x}\right)\right\}\bigg|_{x=L}
\\ &= - h^{2}(t)\ \left[\sum_{i}\sigma_{i}\left(\frac{\alpha_{i}}{2}+ \frac{\gamma_{i}}{(p+1)(p+2)}h^{p}(t)\right)\right]   - \sum_{i}\int_{0}^{L}  \nu_{i} \ |u_{i,x}|^2 \ \diff x   \nonumber
\\ &\phantom{=} -\sum_{i}  \left\{\sigma_{i}\left(\frac{\alpha_{i}}{2} u_{i}^2 + \frac{\gamma_{i}}{p+2} u_{i}^{p+2}\right)- u_{i}\left( \mu_{i}^2 u_{i,xt}+\nu_{i}u_{i,x}\right)\right\}\bigg|_{x=L}.  \label{eqn:L_E_cons}
\end{align}
Since $u$ is bounded on $[0,T]\times X$, $h(t)=u(0,t)$ is bounded on $[0,T]$. Additionally, by Lemma \ref{lemma:nullity} all of the terms in curly braces in \eqref{eqn:L_E_cons} vanish as $L\rightarrow \infty$ uniformly in $t$. Consequently, \eqref{eqn:L_E_cons} indicates that there exists $C\geq 0$ depending on $T$, $\sup_{[0,T]} |h(t)| \ \leq \|u\|_{C\left(0,T, C_{b}\left(X\right)\right)}$, and the coefficients of the PDE such that
\begin{align}\label{eqn:L_E_cons_prime}
\lim_{L\rightarrow \infty} \sum_{i} \int_{0}^{L}  |u_{i}(x,t)|^2+\mu_{i}^2|u_{i,x}(x,t)|^2 \ \diff x \leq \| \varphi \|_{H^{1}\left(X\right)}^2 + C, 
\end{align}
By hypothesis, $u(\cdot,t)\in H^{1}\left(X\right) \ \forall \ t\in [0,T]$. Accordingly, we can go back to \eqref{eqn:L_E_cons} and take $L\rightarrow \infty$ to obtain the formula \eqref{eqn:energy_ode}. 
\end{proof}
Physically, \eqref{eqn:energy_ode} tells us that any change in the solution's energy is due to either viscoelastic damping or movement through the central junction. At first glance, however, it is not clear whether we can expect energy to be gained or lost at the junction. Intuitively, we expect the latter: the real motion of a fluid at such a junction is likely to involve some sloshing against the walls, hence energy is drained from the flow due to friction. For special parameter values and networks we can guarantee that, at the very least, energy is never gained at the junction. Further, this is enough to obtain a solution to gBBMB on $X$ valid for all \emph{positive} times (of course, since we include a dissipative term, we do not intuitively expect even a local solution to exist for negative times). We state these global well-posedness results in the next two corollaries:
\begin{cor}\label{cor:GWP_p_even} If $\sum_{i}\sigma_{i}\alpha_{i}, \ \sum_{i}\sigma_{i}\gamma_{i} \geq 0$, and $p$ is even, then the solution to \eqref{eqn:nl_y_junction_sys} valid up to time $T$ has non-increasing energy, and can be extended to a unique global-in-time solution $u\in C_{b}\left(\left[0, \infty\right), H^{1}\left(X\right)\right)$.
\end{cor}
\begin{proof}
Applying Theorem \ref{thm:energy_X}, we see that 
$$ \diff E / \diff t \leq 0$$ 
if the parameters of the problem are chosen according to the hypothesis. Therefore, we can extend the solution out to a further time $T'>T$ by defining new initial conditions $\tilde{\varphi}\doteq u(x,T) \in H^{1}\left(X\right)\cap \left(C^{2}_{b}(0,\infty)\right)^{N}$ and applying our local well-posedness result once more. Of course, this extended solution has the same regularity as the solution on $[0,T]$. Since $ \diff E / \diff t \leq 0$, Sobolev embedding implies
$$ \|u(x,T) \|_{C_{b}\left(X\right)} \ \lesssim \| \varphi\|_{H^{1}\left(X\right)} < \infty $$
hence $T' > T$.  Since the energy of the solution thus obtained remains non-increasing,  we may iterate the procedure described above as much as we like, obtaining a sequence of existence times tending to $+\infty$ , establishing existence of a global solution.  Uniqueness of the global solution follows from a classical energy argument along the lines of  \cite{BBM1972} \S 4, using the junction conditions as in Theorem \ref{thm:energy_X}. 
\end{proof}
\begin{cor}\label{cor:GWPtwoedge}
If $\sum_{i}\sigma_{i}\alpha_{i}, \ \sum_{i}\sigma_{i}\gamma_{i} =0$, then the solution to \eqref{eqn:nl_y_junction_sys} valid up to time $T$ has non-increasing energy, and can be extended to a unique global-in-time solution $u\in C_{b}\left(\left[0,\infty\right), H^{1}\left(X\right)\right)$.  
\end{cor}
\begin{proof}
Apply the same arguments used to prove Corollary \ref{cor:GWP_p_even}. 
\end{proof}
We emphasize that Corollary \ref{cor:GWPtwoedge} holds regardless of the value of $p$. Note also that, if viscoelasticity is ignored ($\nu_{i}=0 \ \forall \ i$), then energy is actually conserved if the conditions of Corollary \ref{cor:GWPtwoedge} are met. 
\par Are the parameter restrictions imposed by the above corollaries physically meaningful? $\alpha_{i}>0$ is necessary to ensuring long linear waves always move towards $+\infty$ on each edge. Since blood pressure waves are indeed long waves, $\alpha_{i}>0$ is a suitable physical restriction. Recall that we also demand $\gamma_{i}\geq 0$, following the derivation of KdVB and mKdVB by Erbay et al. \cite{EED1992}. Since we do not expect scleroses or stents to cause large changes in the coefficients $\alpha_{i}$ and $\gamma_{i}$ between edges, the constraints $\sum_{i}\sigma_{i}\alpha_{i}\geq0, \ \sum_{i}\sigma_{i}\gamma_{i} \geq 0$ seem to be perfectly reasonable. In particular, these constraints are satisfied in the case $\alpha_{i}\equiv \alpha >0$ and $\gamma_{i}\equiv \gamma \leq 1$. Also, in light of the aforementioned work of Erbay et al., $p=1$ and $p=2$ both correspond to valid asymptotic models of pulsatile flow in viscoelastic tubes. In fact, according to Erbay et al., choosing $p=2$ may in fact be more physically relevant: compared to the $p=1$ model,  the $p=2$ model captures genuinely nonlinear behaviour in a wider variety of viscoelastic materials. Therefore, the hypotheses of both corollaries are definitely of physical relevance.
\begin{remark} A na\"{i}ve first guess as to how \eqref{eqn:energy_ode} can be applied to prove global well-posedness for any odd $p$ fails, as we shall now demonstrate. We may use Sobolev embedding to write
\begin{equation}
\label{eqn:h_embed} |h(t)| \ \lesssim E(t)^{1/2}, 
\end{equation}
hence by \eqref{eqn:energy_ode} we have
\begin{equation}\label{eqn:blowup_ode}
\frac{\diff E}{\diff t} \lesssim E^{p/2+1}. 
\end{equation} 
Since $p\geq 1$, the envelope for energy yielded by the above inequality blows up in finite time. Therefore, the arguments of \cite{BB1973, BonaLuo1995} do not trivially extend to the case of gBBMB on a network for every value of $p$. 
\end{remark}

\section{Numerical Simulations}\label{s:num}
In this section, we describe some elementary numerical simulations of a solitary wave scattering off the junction of a star network with two edges. My intention here is to give an initial push towards understanding the behaviour of the full nonlinear model in some basic test cases, rather than describing a complete and robust numerical scheme. For simplicity, we only consider $p=1$ in the simulations. we focus on determining what initial conditions  and coefficient values $\mu_{i}, \nu_{i}$ allow an anti-solitary wave to be reflected from the junction. 
\par All simulations were written in Python (using Numpy and Scipy). Matplotlib \cite{Matplotlib} and the CMOcean colourmap library \cite{cmocean} were used to make the figures. 

\subsection{Description of Numerical Method}\label{ss:method}
We begin by going over the numerical scheme used to perform the simulations. We work on a finite time interval $[0,T]$ partitioned into uniform intervals of size $\Delta t$. Each edge of our network is identified with $[0,L]$ for some large $L$, and we chop $[0,L]$ into uniform intervals of size $\Delta x$. Superscripts on a function denote a temporal index, and subscripts denote a spatial index. For example, 
$$ u^{n}_{j} \approx u(j\Delta x, n\Delta t). $$
We impose the homogeneous Dirichlet boundary conditions (BCs)
$$u_{i}(t,L)=0 \quad \forall \ i,t.$$
While these BCs keep the numerical routine simple, they also lead to unphysical reflections at the computational boundary. Therefore, our simulations become unreliable once the wave hits the computational boundary. However, they still function well on short-time scales, which is sufficient for understanding wave-junction interactions.  In the future, artificial transparent BCs may be implemented to yield more reliable numerical results, and indeed such BCs have recently been shown to work excellently for the linearized BBM \cite{BMN2018}. 
\subsubsection{Finite Difference Formulation of gBBMB}
In \cite{EM1975,EM1977} Eilbeck and McGuire  investigated several finite difference methods for BBM. They found that the following scheme performed best among those considered: 
\begin{align}
\label{eqn:em_method}
 u^{n+1}_{j-1} - \left(2+\left(\frac{\Delta x}{\mu}\right) ^2 \right)u^{n+1}_{j} +  u^{n+1}_{j+1} & = \frac{\Delta t \Delta x}{\mu^2} \left[\left(\alpha+\gamma u^{n}_{j}\right)\left(u_{j+1}^{n}-u^{n}_{j-1}\right)\right]
 \\ &\phantom{=} + u^{n-1}_{j-1} - \left(2+\left(\frac{\Delta x}{\mu}\right) ^2 \right)u^{n-1}_{j} +  u^{n-1}_{j+1}
\end{align}
This amounts to using a leapfrog approximation of $u_{t}$ and $u_{x}$ and a centred approximation of $1-\mu^2\partial_{x}^2$.  From \cite{EM1975,EM1977}, the above scheme boasts the following features: 
\begin{itemize}
\item  The scheme is second order in both space and time, and stable provided the solution is roughly $\mathcal{O}(1)$ or less.
\item While the scheme is implicit, it does not require the solution of a nonlinear system at each time step. Further, since $A$ only depends on $\mu$ and $\Delta x$, it can be factorized in a pre-processing stage for more efficiency.
\item The scheme preserves the solitary wave solutions of BBM quite well. 
\end{itemize}
To accommodate the diffusive term as naturally as possible, we use the Crank-Nicolson leap-frog time-stepping method for advection-diffusion equations  \cite{ARW1995}. The resulting scheme is second-order in space and time. 
\subsubsection{Adding the Interface Condition}
Now, we illustrate how to add the mass conservation interface condition into the numerical solver. We sample the network at $J+1$ points (including the computational boundaries, where we impose homogeneous Dirichlet BCs) and let $j$ is the spatial index where the interface lies. We then treat the solution as an array 
$$
\begin{bmatrix}
u^{n}_{0} &  u^{n}_{1} & \dots & u^{n}_{j-1} & u^{n}_{j}  & u^{n}_{j+1}  & \cdots & u^{n}_{J-1} &  u^{n}_{J} 
\end{bmatrix}^{\T},
$$
thus we can extract the solution values on each edge according to
\begin{align*}
u_{1}(t^{n}) &\approx \begin{bmatrix}
u^{n}_{0} &  u^{n}_{1} & \dots & u^{n}_{j-1} & u^{n}_{j} \end{bmatrix}^{\T},
\\
u_{2}(t^{n}) &\approx \begin{bmatrix}
 u^{n}_{j}  & u^{n}_{j+1}  & \cdots & u^{n}_{J-1} &  u^{n}_{J} 
 \end{bmatrix}^{\T}.
\end{align*} 
To couple the edges together we must discretize the interface condition \eqref{eqn:MC}. Following the discussion in the previous subsection, it is most sensible to use a leapfrog discretization of the time derivative in \eqref{eqn:MC}, and forward finite differences to discretize the space derivatives. These choices give rise to the linear algebraic equation 
\begin{align*}
\frac{\mu_{1}^2 \left(u^{n+1}_{j-1}-u^{n-1}_{j-1}\right) - \left(\mu_{1}^2 +\mu_{2}^2 \right)\left(u^{n+1}_{j}-u^{n-1}_{j}\right) +\mu_{2}^2 \left(u^{n+1}_{j+1}-u^{n-1}_{j+1}\right)}{2\Delta x\Delta t} &= f_{2}(u^{n}_{j}) - f_{1}(u^{n}_{j})
\\
&\phantom{=} -\frac{\nu_{1} u_{j-1}^{n} -\left(\nu_{1}+\nu_{2}\right)u_{j}^{n}+\nu_{2}u_{j+1}^{n}}{\Delta x}
 \end{align*}
as a discrete substitute for \eqref{eqn:MC}. Note that this means our scheme drops from a second order method to a first order one in a neighbourhood of the interface.

\subsection{Test Cases for a $2$-Edge Network} \label{ss:2edge_sim}
Now, we use the finite difference scheme discussed above to simulate the solution to gBBMB on a network with two edges, each of length $100$ units. Since we are interested in modelling blood pulses, in all of our trials we consider an initial state given by a solitary wave. gBBM possesses the following solitary wave solutions, parameterized by speed $c \in (\alpha, \infty)$ and initial peak location $x_{0}\in \mathbb{R}$:
\begin{subequations}\label{eqn:sw_data}
\begin{align}
W &= \frac{p}{2\mu} \sqrt{1-\frac{\alpha}{c}},
\\
A&= \frac{p}{2\mu W} \sqrt{\frac{2\gamma}{c(p+1)(p+2)}},
\\
u(x,t) &=\left[A\cosh\left(\frac{x-x_{0}-ct}{W}\right)\right]^{-2/p}.
\end{align}
 In each test case, we start with a solitary wave on the incoming edge moving towards the junction.
\par As a concrete measure of our scheme's performance, we look at how well the simulation conserves the mass of the initial solution. We denote the mass of our solution by 
$$ M(t) \doteq  \sum_{i}\int_{e_{i}}u(x,t) \ \diff x, $$
and the percent relative error in mass by
$$ \delta M \doteq 100\ \frac{|M(t)-M(0)|}{M(0)} \ \%. $$
\end{subequations}
For all of the test cases presented in the sequel, we either plot $\delta M$ as a function of time or report its maximum. 

\subsubsection{Small-amplitude Initial Data with Variable Linear Elasticity}\label{ss:test1}
First, we investigate the scattering of a solitary wave with parameters $c=2, x_{0}=60$ as a result of moving between edges with dispersion coefficients satisfying $\mu_{1}=1$ and $\mu_{2}> 1$. The power of the nonlinear term is $p=1$, and the other coefficients of gBBMB are fixed at $\nu_{i}=0$ and $\alpha_{i}=\gamma_{i}=1$. Physiologically, this could correspond to a blood pulse moving from an unhealthy, sclerotic arterial segment to a healthy or stented arterial segment. For these parameter values, we have that the solitary wave's amplitude is $3$. The discretization parameters are $\Delta x=\Delta t=0.025$.
\par Figure \ref{fig:test1_hov} shows filled space-time contour plots of the solution to gBBMB in two cases: $\mu_{2}=1.1$ and $\mu_{2}=1.5$ (recall that the interface between edges is placed at $x=100$ here).  In the first subfigure, the only change to the waveform is the creation of a very small disturbance emerging from the junction in the wake of the wave. In the second subfigure, the wave is noticeably modified by its interaction with the interface, becoming shorter, wider, and slower as it moves into the second edge.  We intend to empirically determine a relationship between the speeds of the incident and transmitted solitary waves using improved numerical methods in future work (see also \cite{Cascaval2012}).  

\par Finally, Figure \ref{fig:test1_mass} shows the precent relative error in total mass for both of these trials. In either case, the error is well below $1\%$, implying that our computer results are indeed physical. Note that the error is largest around time $t=20$, which by inspection of Figure \ref{fig:test1_hov} is precisely when the wave crosses the interface. This is explained by recalling that our scheme is only \emph{first order} in space near $x=100$: more error creeps in when $u(x,t)$ is nonzero in this region. 
\begin{figure}
\vspace{-1cm}
\centering 

\hspace{-2.5cm}
\begin{subfigure}{0.55\textwidth}
\centering
\includegraphics[width=1.3\textwidth]{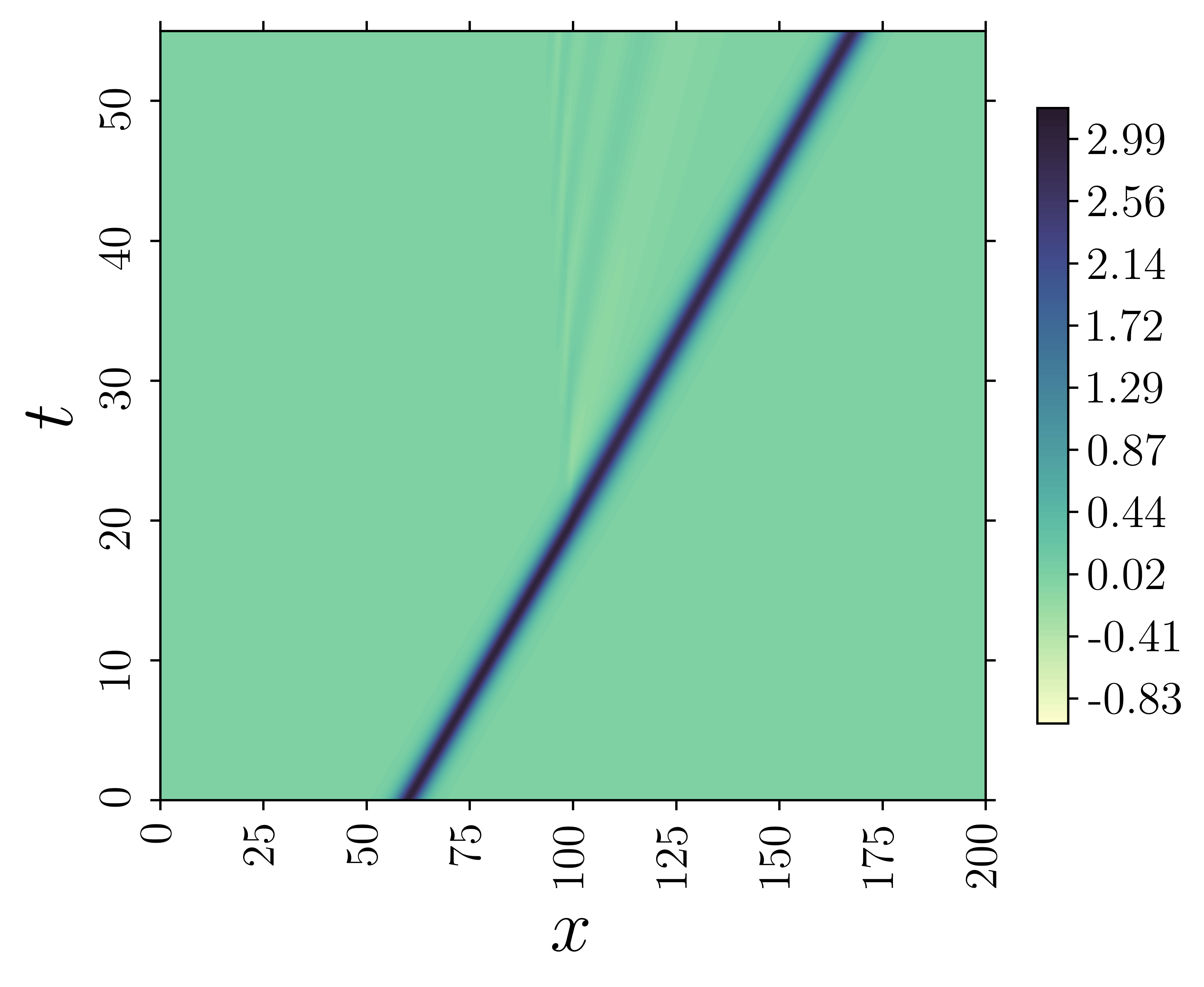}
\caption{$\mu_{1}=1, \mu_{2}=1.1$.}
\end{subfigure}

\hspace{-2.5cm}
\begin{subfigure}{0.55\textwidth}
\centering
\includegraphics[width=1.3\textwidth]{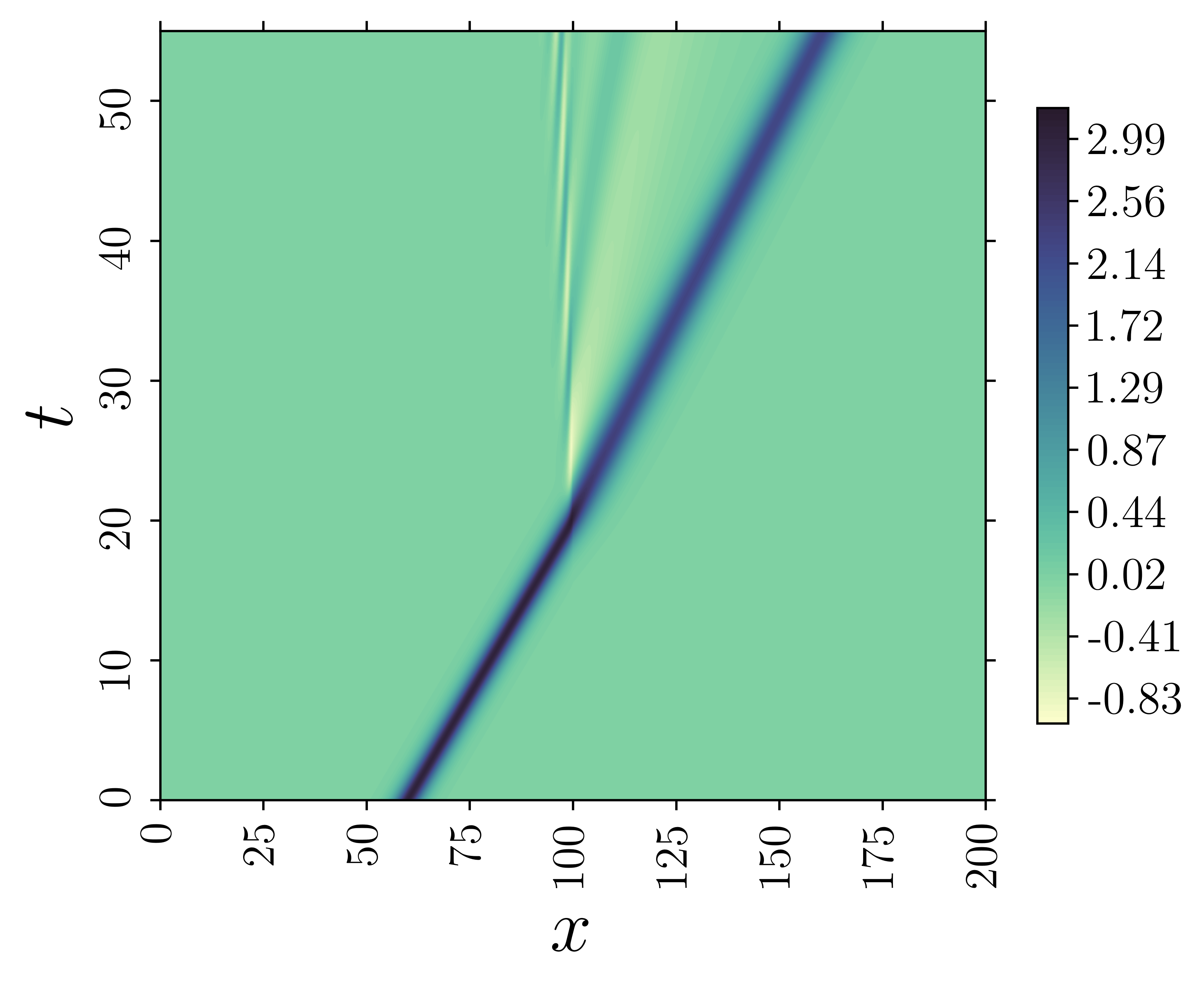}
\caption{ $\mu_{1}=1, \mu_{2}=1.5$. }

\end{subfigure}
\caption{Initial condition is a solitary wave with $c=2, x_{0}=60$, $\nu_{i}=0$, and $\gamma_{i}=1$. The interface is at $x=100$.}
\label{fig:test1_hov}
\end{figure}

\begin{figure}
\hspace{-2cm}
\centering 
\begin{subfigure}{0.45\textwidth}
\includegraphics[width=1.2\textwidth]{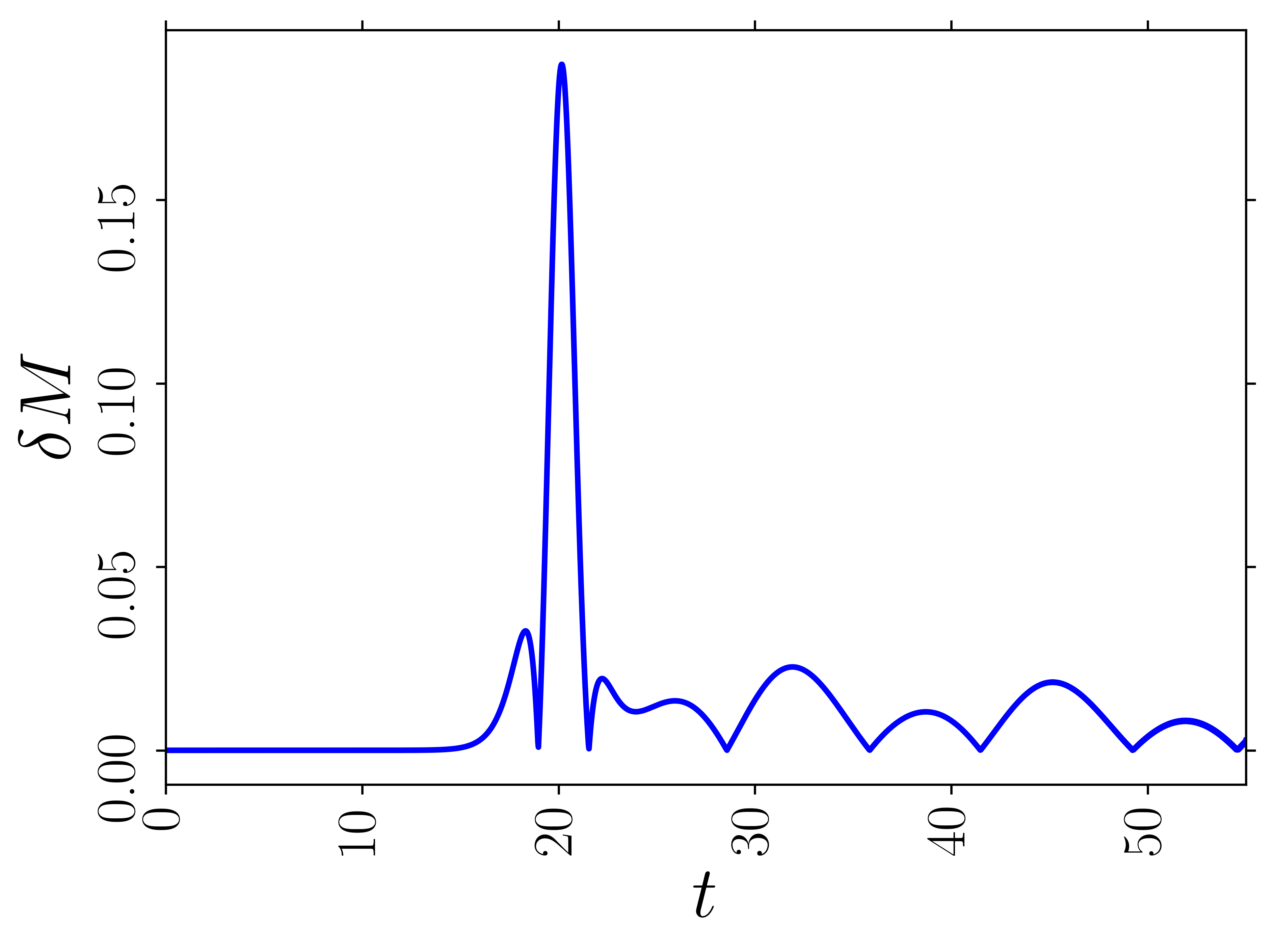}
\caption{$\mu_{1}=1, \mu_{2}=1.1$.}
\end{subfigure}
~~~~~~~~~~
\begin{subfigure}{0.45\textwidth}
\includegraphics[width=1.2\textwidth]{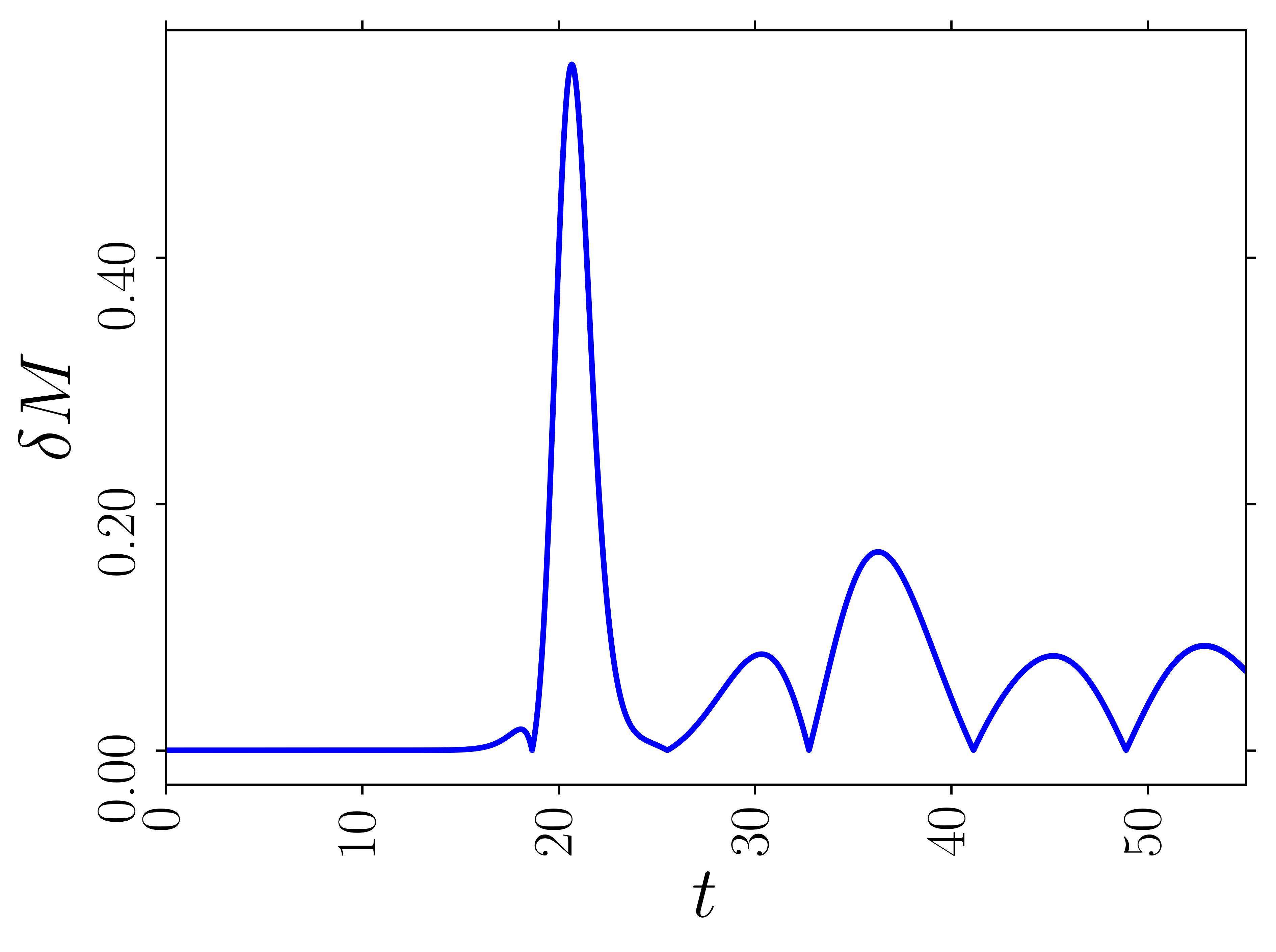}
\caption{$\mu_{1}=1, \mu_{2}=1.5$. }

\end{subfigure}
\caption{Percent relative error in mass for our numerical solution with $c=2, x_{0}=60, \nu_{i}=0,$ and $\gamma_{i}=1$. }
\label{fig:test1_mass}
\end{figure}

\subsubsection{Large-amplitude Initial Data with Variable Linear Elasticity}
Now, we look at the same test case from the previous subsection but with the speed increased to $c=5$. This means that the amplitude of the initial waveform is $12$ units, four times larger than the waves in Figure \ref{fig:test1_hov}.  Figure \ref{fig:test2_hov} displays the results of two simulations with different $\mu_{2}$. In Figure \ref{fig:test2_hov} (a),  $\mu_{2}=1.1$ and we see that, as in the analogous case in the previous subsection, the wave is mostly unperturbed by the interface. For Figure \ref{fig:test2_hov} (b), $\mu_{2}=1.5$ and the picture is very different: a slow anti-solitary wave is reflected from the interface. Thus we have found that the existence of a reflected wave depends on the amplitude of the incident wave as well as the ratio $\mu_{2}/\mu_{1}$.  we remark that, based on other numerical experiments we have completed, $\mu_{2}=1.5$ seems to be the smallest $\mu_{2}$ for which a reflected wave is visible when $c=5$. 
\par The relative error in mass $\delta M$ is less than $0.5\%$ for these test cases. we have omitted showing plots of $\delta M(t)$ here because they are very similar to Figure \ref{fig:test1_mass}: the biggest spike in error occurs when the wave crosses the junction. 
\begin{figure}
\vspace{-1cm}
\centering 

\hspace{-2.5cm}
\begin{subfigure}{0.55\textwidth}
\centering
\includegraphics[width=1.3\textwidth]{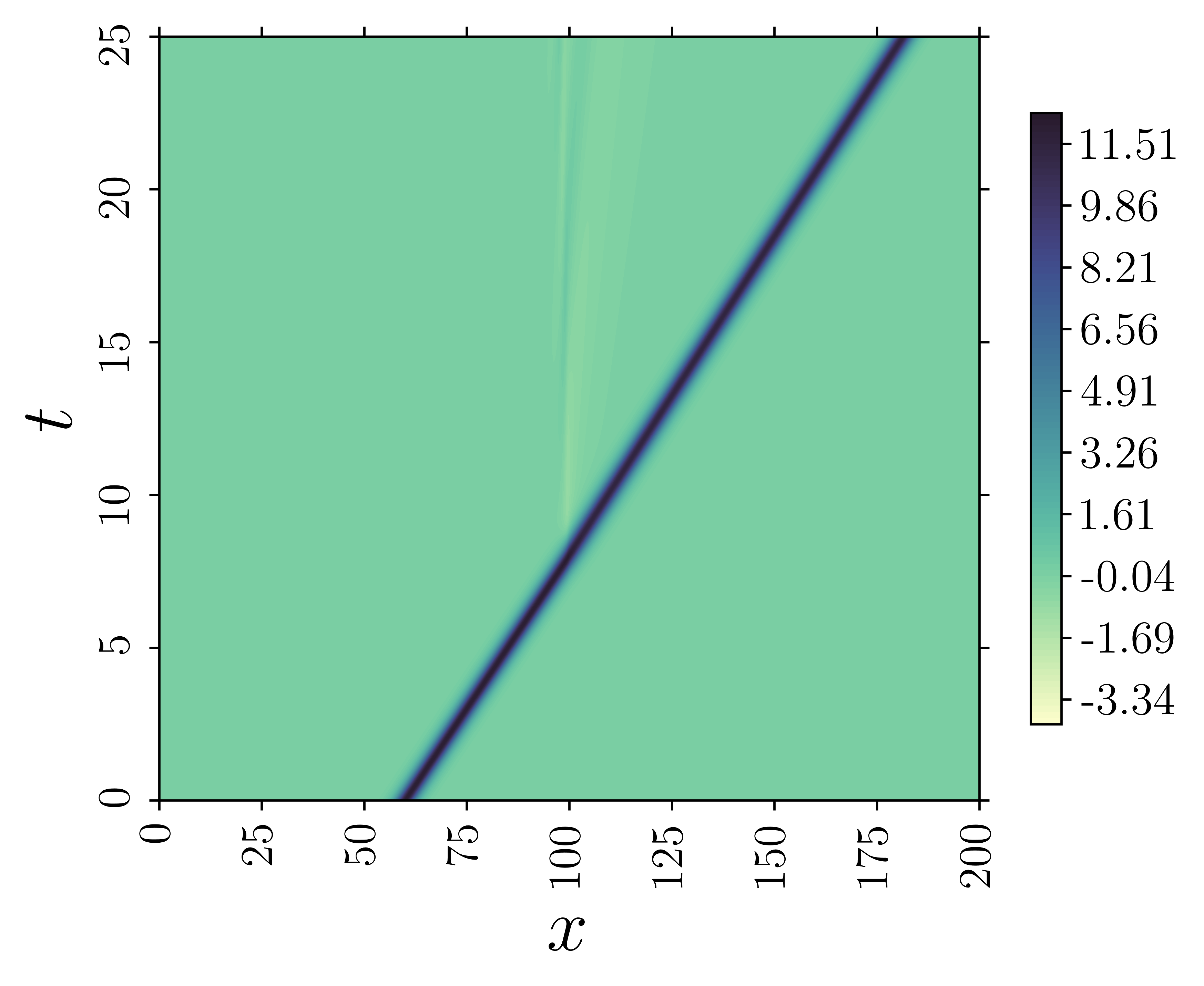}
\caption{$\mu_{1}=1, \mu_{2}=1.1$.}
\end{subfigure}

\hspace{-2.5cm}
\begin{subfigure}{0.55\textwidth}
\centering
\includegraphics[width=1.3\textwidth]{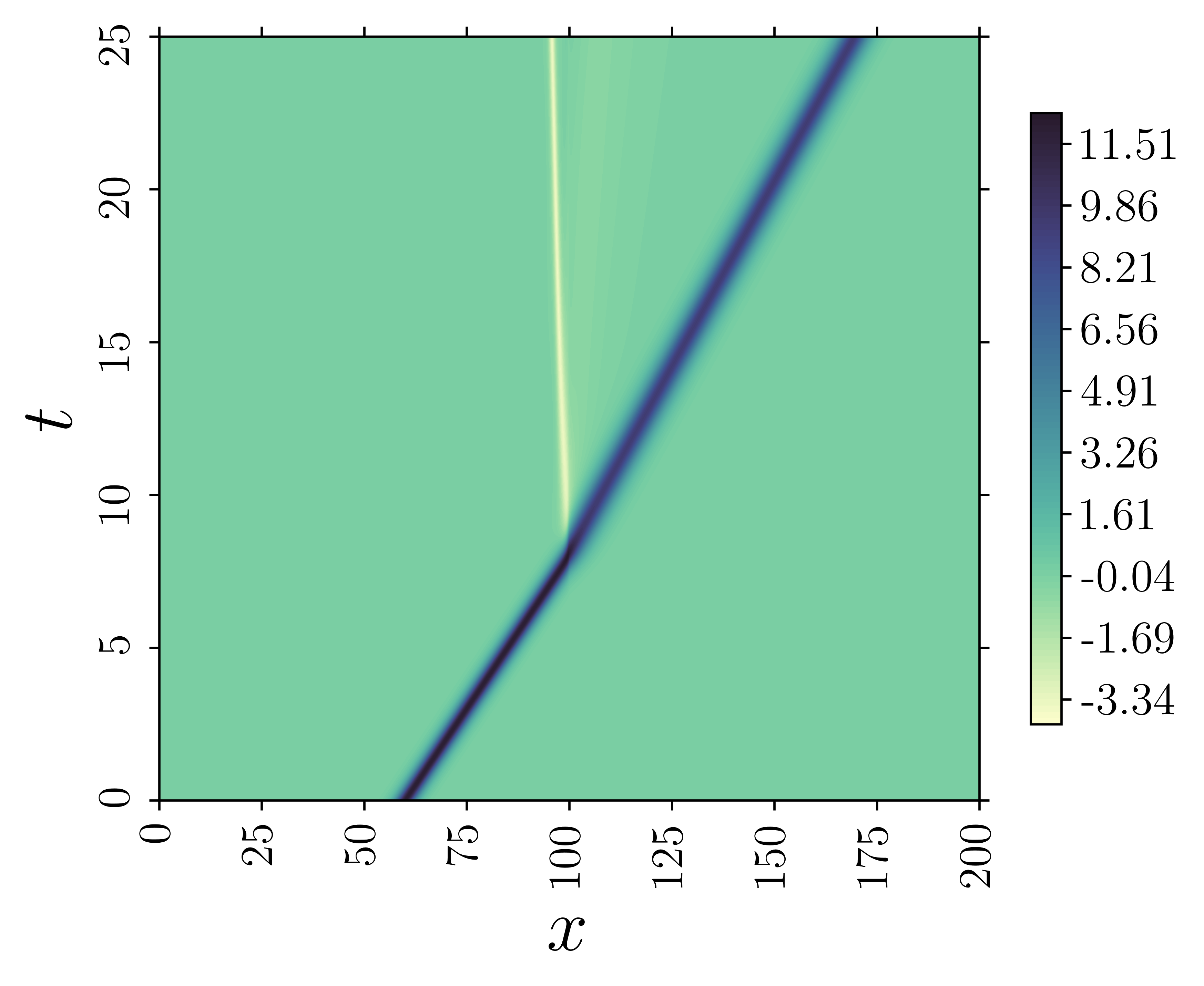}
\caption{ $\mu_{1}=1, \mu_{2}=1.5$. }

\end{subfigure}
\caption{Initial condition is a solitary wave with $c=5, x_{0}=60$. $\nu_{i}=0$ and $\gamma_{i}=1$. The interface is at $x=100$. Note that a slow anti-solitary wave is reflected from the interface in (b), while no reflection occurs in (a). This suggests that reflection depends both on wave amplitude and the ratio $\mu_{2}/\mu_{1}$.}
\label{fig:test2_hov}
\end{figure}



\subsubsection{Large-amplitude Initial Data with Variable Linear Elasticity and Dissipation}
We now investigate how viscoelasticity affects the wave-junction interaction.  Again, we study a solitary wave with $c=5$ and $x_{0}=60$ and suppose that $\alpha_{i}=\gamma_{i}=1$ and $\mu_{1}=1$. This time, we also take $\nu_{1}=1, \nu_{2} =0.1$ and vary $\mu_{2}\geq \mu_{1}$.  Finally, the discretization parameters are $\Delta x=\Delta t=0.025$.
\par The results of two different numerical tests are shown in Figure \ref{fig:test3_hov}. In the first subplot, we have set $\mu_{2}=\mu_{1}=1$, so the damping parameter $\nu_{i}$ is the only coefficient that changes across the interface. Initially, we see that viscoelasticity changes the shape of the solitary wave quite drastically, gradually slowing it down and giving it a long and nearly flat ``tail". The highest peak of the solution ends up being transmitted through the junction as a solitary wave, and as the long flat tail moves through the junction it creates more solitary waves over time. Due to the low damping coefficient in the outgoing edge, these solitary waves essentially preserve their shape and speed over the remainder of the simulation. In Figure \ref{fig:test3_hov} (b), we consider the case where $\mu_{2}$ is increased to $\mu_{2}=1.5$. Recall from Figure \ref{fig:test2_hov} that, when $\nu_{i}=0$, an anti-solitary wave was reflected from the interface. Here, however, there is no reflection at all. Indeed, the only change from Figure \ref{fig:test3_hov} is the speed and width of the transmitted waves, which we have come to expect by now. Thus we have found that viscoelasticity prevent wave reflection off the junction. 
\par We have also plotted $\delta M$ for these viscoelastic simulations in Figure \ref{fig:test3_mass}. The mass error is no larger than $0.2\%$, again indicating a satisfactory level of discrete mass conservation. Note that the graph of $\delta M$ does not drop back down after the initial wave-interface collision, due to the long tail the solitary wave develops. 
\begin{figure}
\vspace{-1cm}
\centering 

\hspace{-2.5cm}
\begin{subfigure}{0.55\textwidth}
\centering
\includegraphics[width=1.3\textwidth]{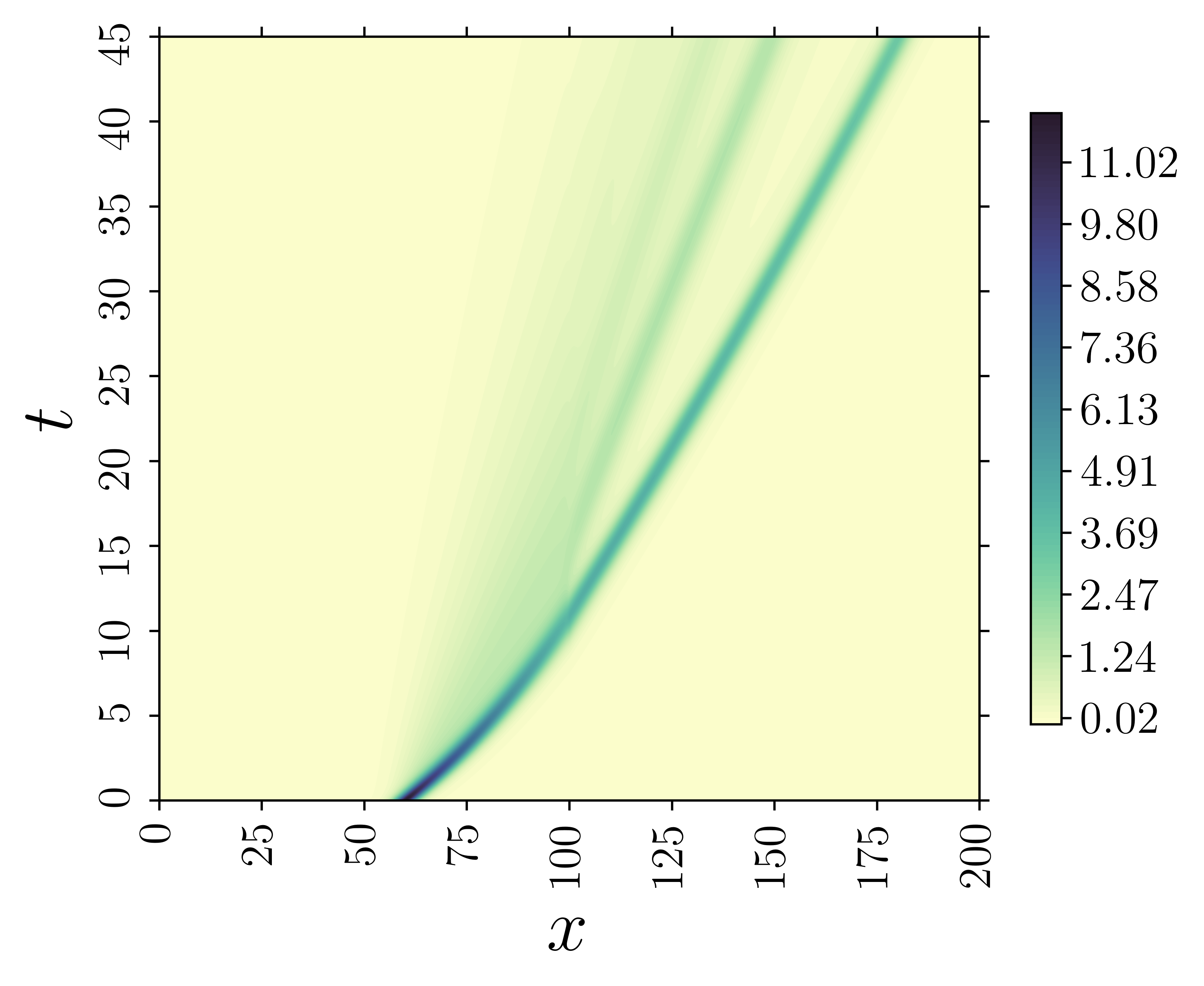}
\caption{$\mu_{1}=\mu_{2}=1, \nu_{1}=1, \nu_{2}=0.1$.}
\end{subfigure}

\hspace{-2.5cm}
\begin{subfigure}{0.55\textwidth}
\centering
\includegraphics[width=1.3\textwidth]{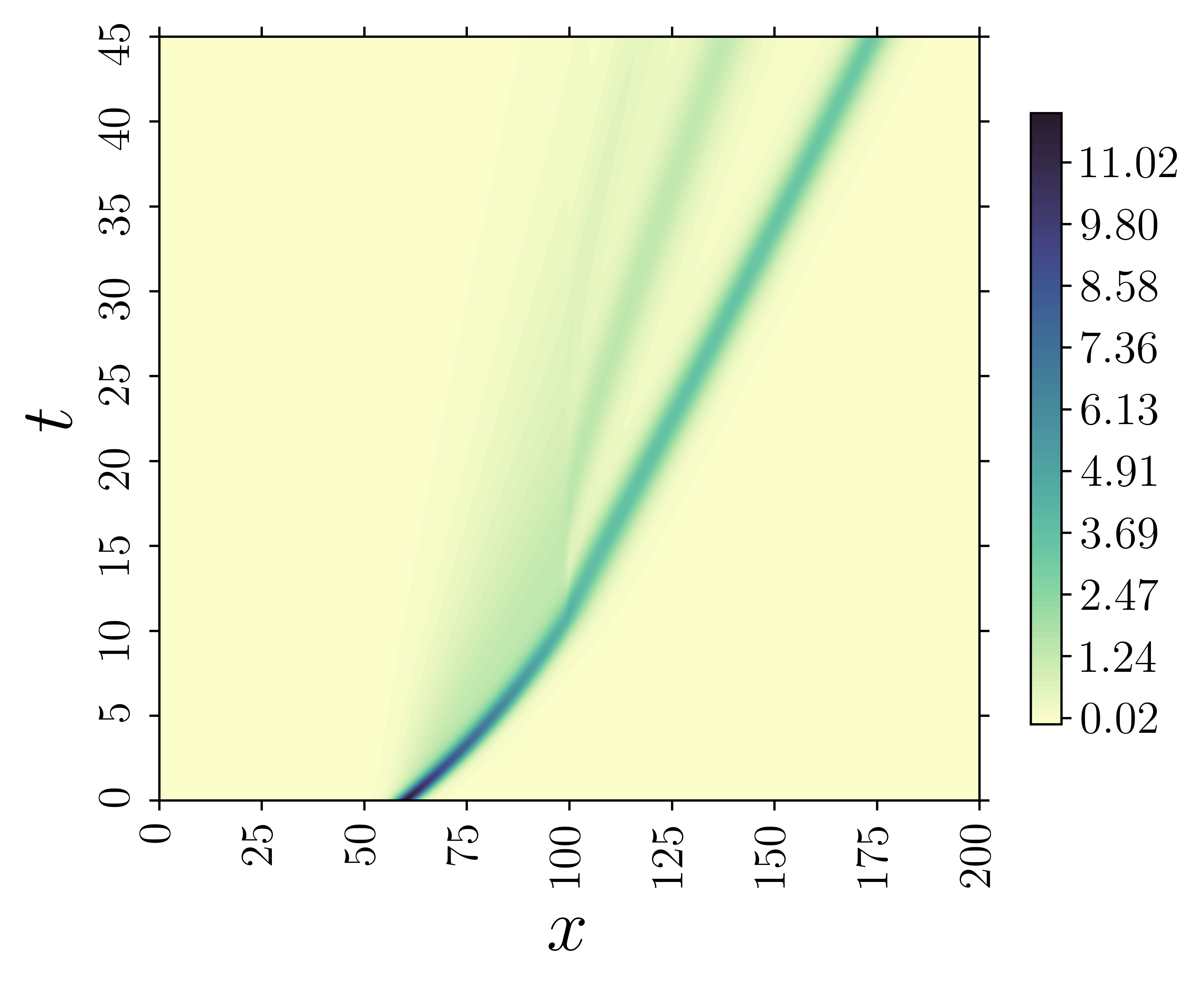}
\caption{$\mu_{1}=1, \mu_{2}=1.5, \nu_{1}=1, \nu_{2}=0.1$.}

\end{subfigure}
\caption{Initial condition is a solitary wave with $c=5, x_{0}=60,  \gamma_{i}=1$. The interface is at $x=100$.}
\label{fig:test3_hov}
\end{figure}

\begin{figure}
\hspace{-2cm}
\centering 
\begin{subfigure}{0.45\textwidth}
\includegraphics[width=1.2\textwidth]{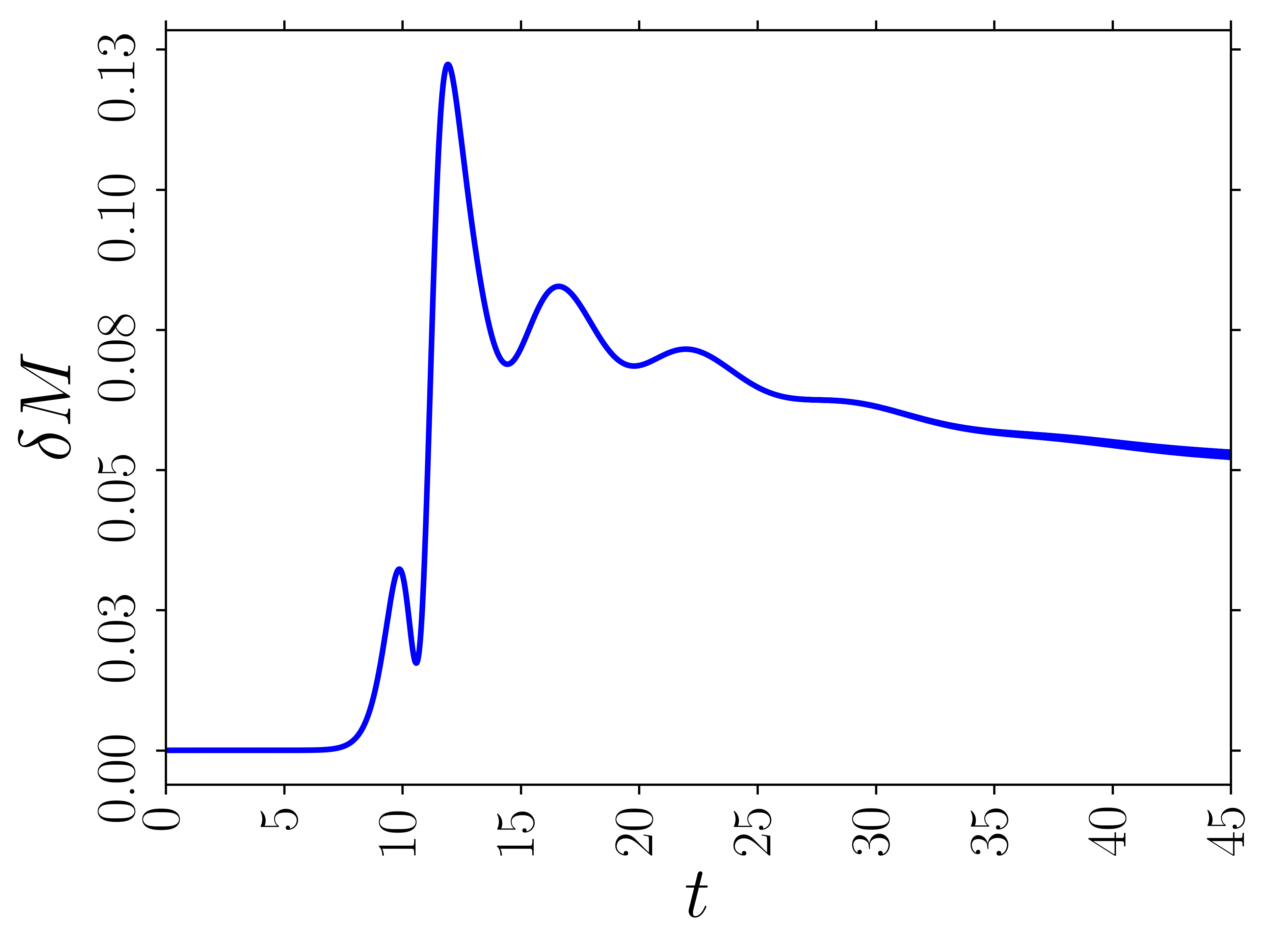}
\caption{$\mu_{1}=\mu_{2}=1, \nu_{1}=1, \nu_{2}=0.1$.}
\end{subfigure}
~~~~~~~~~~
\begin{subfigure}{0.45\textwidth}
\includegraphics[width=1.2\textwidth]{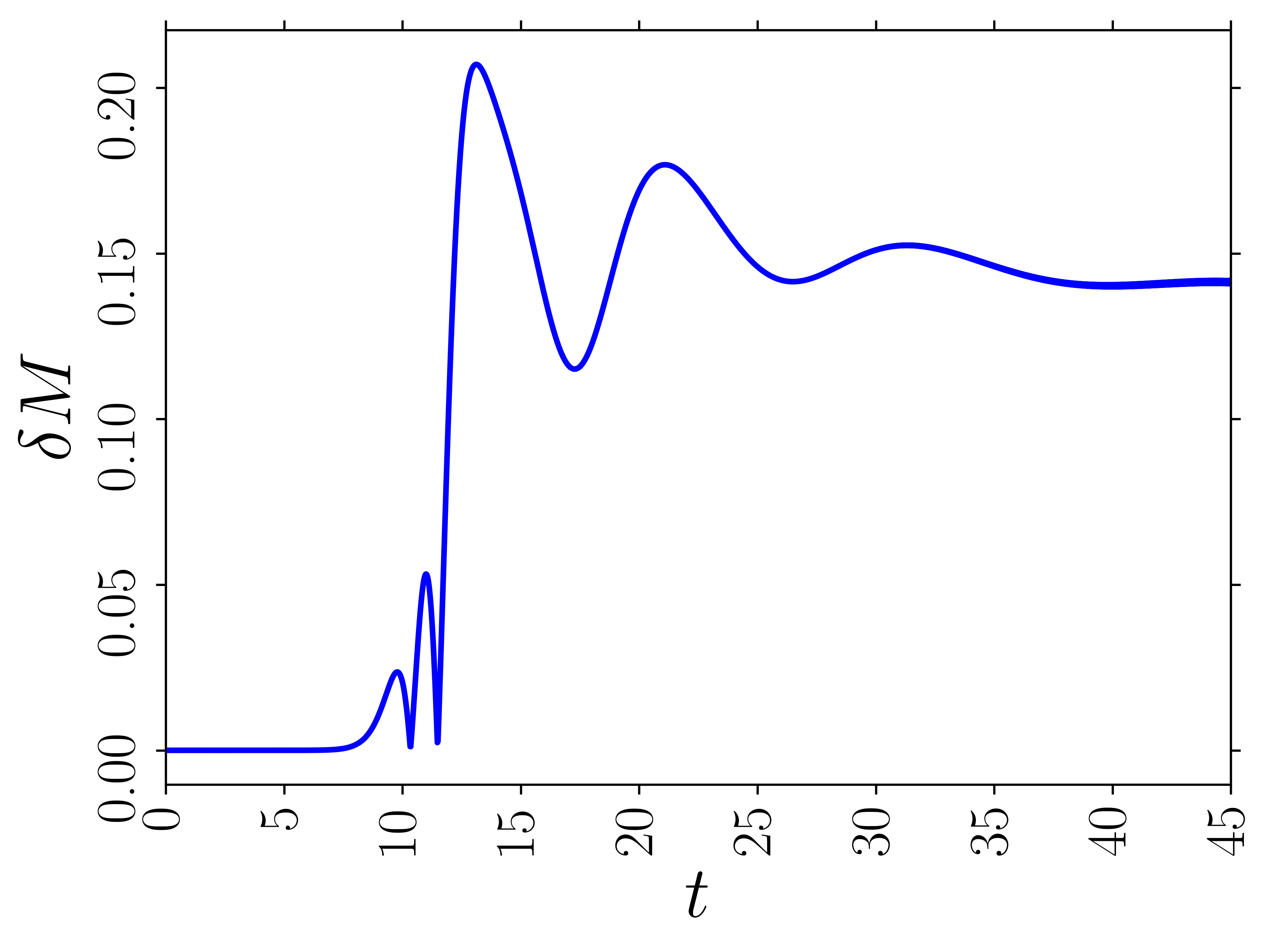}
\caption{$\mu_{1}=1, \mu_{2}=1.5, \nu_{1}=1, \nu_{2}=0.1$.}

\end{subfigure}
\caption{Percent relative error in mass for our numerical solution with $c=5, x_{0}=60, \gamma_{i}=1$. }
\label{fig:test3_mass}

\end{figure}

\section{Acknowledgements} 
The author would like to extend thanks to Fabio Pusateri and the members of his advisory committee for useful discussions. 

\bibliography{stents_math.bib,stents_med.bib,1D_bloodflow.bib,numerical.bib,BBM.bib,networks_math.bib,KdV.bib}{} 
\bibliographystyle{plain}

\end{document}